\documentclass[11 pt]{article}
\usepackage{amsmath, amssymb,amsthm, setspace, palatino, wrapfig, caption, subfig, float}
\usepackage{graphicx}
\usepackage[all]{xy}
\theoremstyle{plain}
\newtheorem{thm}{Theorem}[section]
\newtheorem{mainthm}[thm]{Main Theorem}
\newtheorem{lemma}[thm]{Lemma}
\newtheorem{minilemma}[thm]{Mini-Lemma}
\newtheorem{prop}[thm]{Proposition}
\newtheorem{cor}[thm]{Corollary}
\newtheorem{maincor}[thm]{Main Corollary}
\newtheorem{defn}[thm]{Definition}
\newtheorem{append}[section]{Appendix}
\theoremstyle{definition}
\newtheorem{example}[thm]{Example}
\newcommand{\brac}{\langle{k}\rangle}

\newcommand{\Cob}{\text{Cob}_{d,\brac}^{\theta}}
\newcommand{\thomN}{B^{{{\theta}_N^*{\gamma}{^\perp}}(d,n)}}
\newcommand{\thom}{B^{-{\theta}^*{\gamma}_d}}
\newcommand{\plainthom}{G^{{-{\gamma}_d}}}
\newcommand{\plainthomN}{G^{{{\gamma}{^\perp}}(d,n)}}
\newcommand{\kzerospec}{{\Omega}^{{\infty}-1}_{\brac}\thom\brac}
\newcommand{\kzerohocolim}{{\Omega}^{{\infty}-1}\underset{\underline{2}^k_*}{\text{hocolim}}\;\thom\brac_*}
\newcommand{\struc}{B_{d+n}\brac}
\newcommand{\strucmap}{\theta}
\newcommand{\strucmapN}{{\theta}_N}
\newcommand{\strucmapNone}{{\theta}_{N+1}}
\newcommand{\fullstruc}{(B,{\sigma}',\theta)}
\newcommand{\sheafC}{C_{d,\brac}^\theta}

\newcommand{\sheafCtr}{C_{d,\brac}^{\theta,tr}}

\newcommand{\sheafD}{D_{d,\brac}^\theta}
\newcommand{\sheafDn}{D_{d,n,\brac}^\theta}
\newcommand{\sheafDtr}{D_{d,\brac}^{\theta,tr}}

\newcommand{\bighoco}{\underset{\underline{2}^{k-1}_*}{\text{hocolim}}\;\overline{{\partial}_k}E_*}
\newcommand{\bighocothom}{\underset{\underline{2}^{k-1}_*}{\text{hocolim}}\;\overline{{\partial}_k}\plainthom_*}
\newcommand{\smallhoco}{\underset{\underline{2}^{k-1}_*}{\text{hocolim}}\;{\partial}_kE_*}
\newcommand{\smallhocothom}{\underset{\underline{2}^{k-1}_*}{\text{hocolim}}\;{\partial}_k\plainthom_*}

\newcommand{\geocanon}{{\gamma}({d},{n})\brac}
\newcommand{\geoperpcanon}{{\gamma}^{\perp}({d},{n})\brac}
\newcommand{\twokob}{a\hspace{1pt}\in\hspace{1pt}\underline{2}^k}
\newcommand{\twokmor}{a<b\hspace{1pt}\in\hspace{1pt}\underline{2}^k}
\newcommand{\keucNless}{\mathbb{R}^k_+{\times}\mathbb{R}^{d+n-1}}

\title{Cobordism Categories of Manifolds with Corners}
\author{Josh Genauer}
\date{\today}

\begin{document}
\maketitle

\begin{abstract}
   In this paper we study the topology of cobordism categories of manifolds with corners.  Specifically, if $\text{Cob}_{d,\brac}$ is the category whose objets are a fixed dimension d, with corners of codimension $\leq\;k$,  then we identify the homotopy type of the classifying space $B\text{Cob}_{d,\brac}$ as the zero space of a homotopy colimit of  certain diagram of Thom spectra.  We also identify the homotopy type of the corresponding cobordism category when extra tangential structure is assumed on the manifolds.   These results   generalize  the results  of Galatius, Madsen, Tillmann and Weiss $\cite{GMTW}$, and their proofs are an adaptation of the methods of $\cite{GMTW}$.  As an application we describe the homotopy type of the category of open
and closed strings with a background space X,  as well as its higher dimensional
analogues.  This generalizes work of  Baas-Cohen-Ramirez $\cite{BCR}$ and Hanbury $\cite{LizsThesis}$.
\end{abstract}
\begin{center}
$$\text{Mathematics Subject Classification 57R90, 57R19, 55N22, 55P47}$$
\end{center}
\section{Introduction}

$\indent$Cobordism categories of manifolds with corners are of interest to both mathematicians and physicists.  These categories are particularly relevant when studying open-closed topological and conformal field theories as such field theories are monoidal functors from these cobordism categories to categories such as vector spaces, chain complexes or other symmetric monoidal categories.   Much work has already been done in this setting.  See for example Moore $\cite{Moore}$, Segal $\cite{Segal}$, Costello $\cite{Costello}$, Baas-Cohen-Ramirez $\cite{BCR}$, and Hanbury $\cite{LizsThesis}$.\\

The main result of our paper, is a formula for calculating the homotopy type of the classifying space of the cobordism category $\text{Cob}_{d,\brac}$ of manifolds of fixed dimension $d$ with corners of codimension ${\leq}k$.  The result is the zero space of a homotopy colimit over a certain diagram of Thom spectra.  More generally, let $\theta:B\rightarrow{BO}(d)$ be a fibration.  Then this recipe also allows us to compute the homotopy type of $\Cob$, the cobordism category whose morphisms are cobordisms with corners together with structure on the tangent bundle determined by $\theta$.  The precise definition of these categories will be given in $\textbf{section 4}$.\\
  
In some interesting cases we are able to evaluate this homotopy colimit explicitly.  A fibration of particular interest is  $BSO(d){\times}X{\rightarrow}BO(d)$, which is the composition
$$BSO(d){\times}X\overset{\text{project}}{\longrightarrow}{BSO(d)}\overset{\text{include}}{\longrightarrow}{BO(d)}.$$
The corresponding tangential structure is an orientation on the manifold, together with a map to a background space $X$.  We will call this category $\text{Cob}_{d,\brac}(X)$ and we prove that $B\text{Cob}_{d,\brac}(X)$  shares its homotopy type with $X_+{\wedge}B{Cob}_{d,\brac}$, a generalized homology theory of $X$.  See the third and last part of $\textbf{section 6}$ for proof.  The category $\text{Cob}_{d,\brac}(X)$ is particularly important in the 2-dimensional case where this category of ``surfaces in $X$" plays a role in both Gromov-Witten theory as well as string topology.  Variations on this category have been studied in several contexts, see Baas $\cite{Baas}$, Sullivan $\cite{Sull}$, Cohen-Madsen $\cite{Co-Mad}$, and Hanbury.\\

The kind of corners on manifolds we consider have been studied by J\"{a}nich $\cite{Janich}$ and more recently by $\cite{Laures}$.  The main feature of the corner structure under consideration is, as Laures showed, that such manifolds embed in $\mathbb{R}^k_+{\times}\mathbb{R}^N$ in such a way that the corner structure is preserved.  Here $\mathbb{R}_+$ is the nonnegative real numbers and $\mathbb{R}^k_+$ is the cartestian product $(\mathbb{R}_+)^k$.  In turn, the corner structure induces a stratification of the underlying manifold by nested submanifolds whose inclusion maps into one another form a cubical diagram.  In this way these manifolds with corners naturally live in the category of cubical diagrams of spaces which Laures referred to as $\brac$-spaces.  Laures originally found cobordisms of $\brac$-manifolds while investigating Adams-Novikov resolutions, but $\brac$-manifolds have appeared in other contexts as well.  They appear in open-closed string theory mentioned above.  R. Cohen, Jones, and Segal found that the space of gradient flow lines of a Morse function on a closed manifold $M$ naturally has the structure of a framed manifold with corners $\cite{CJS}$.  R. Cohen has used these ideas in his work on Floer theory $\cite{Cohen}$.  Lastly, $\brac$-manifolds naturally appear in the author's future work on stable automorphism groups of closed manifolds.\\

Our paper is also heavily indebted to the work of Galatius, Madsen, Tillmann, and Weiss $\cite{GMTW}$, who calculated the homotopy type of the cobordism category of closed manifolds as the zero space of a Thom spectrum.  This paper is a generalization of their result and uses their techniques.  From this perspective, the cost for considering manifolds with corners is that the classifying space of the cobordism category has the homotopy type of the zero space of a homotopy colimit over a cubical diagram of Thom spectra instead of a single spectrum.\\

The paper is organized as follows.  In $\textbf{section 2}$ we discuss the cubical diagrams of spaces which, in the tradition of Laures's paper, we call $\brac$-spaces.   In $\textbf{section 3}$ we proceed on to vector bundles in the $\brac$-space setting, and we introduce the analogue of prespectra in category of diagrams.  In $\textbf{section 4}$ we define cobordism categories of manifolds with corners and state the main theorem of the paper.  We prove the main theorem in $\textbf{sections 5}$.  In $\textbf{section 6}$ we describe some applications including an explicit description of the homotopy type of $\Cob$ as the zero space of well known spectra.  $\textbf{Section 7}$ is an appendix of some results from differential topology and bundle theory modified to the $\brac$-manifold setting.  We need these facts for the proof in $\textbf{section 5}$ and are used nowhere else.\\

The list of people the author would like to thank possibly stretches longer than this paper, so we restrict attention to three dear people.  The author thanks David Ayala for many interesting conversations.  The author thanks Soren Galatius for patiently explaining cobordism categories amongst other things and for his laconic, earnest support.  Lastly, the author thanks Ralph Cohen, wise in things mathematical and otherwise.  Without Prof. Cohen's sage guidance, this project would not exist. \\

The results of this paper are part of the author's Ph.D. thesis, written under the direction of Prof. Ralph Cohen at Stanford University.  
\newline
\newline

\section{$\brac$-manifolds and their embeddings}
$\indent$In this section we give an intrinsic definition of $\brac$-manifolds and discuss some of their basic properties.  Most importantly, we explain what we mean by an embedding of a $\brac$-manifold and show that the space of embeddings of any $\brac$-manifold into (theorem \ref{mainthm2}).  This generalizes the result of Laures which showed that the space of such embeddings is nonempty.  Next we shall introduce $\brac$-spaces, which are (hyper)cubical diagrams of spaces.  Just as smooth manifolds form a subcategory of topological spaces so too do $\brac$-manifolds naturally sit inside the category of $\brac$-spaces.  In fact, all the usual objects in differential topology such as tangent bundles, normal bundles, embeddings, and Pontrjagin-Thom constructions have analogues in the setting of $\brac$-spaces.\\
 
  As motivation, we describe a way to generate examples of $\brac$-manifolds.  To this end, let $a=(a_1,...,a_k)$ be a k-tuple of binary numbers and let $\mathbb{R}^k(a)$ be the subspace of $\mathbb{R}^k$ consisting of k-tuples $(x_1,...,x_k)$ such that $x_i=0$ if $a_i=0$.  Now suppose we have an embedded closed submanifold $$N\hspace{3pt}\hookrightarrow\hspace{3pt}\mathbb{R}^k{\times}\mathbb{R}^n$$
 that is transverse to $\mathbb{R}^k(a){\times}\mathbb{R}^n$ for all possible k-tuples $a$.  As pointed out in $\cite{Laures}$, 
the intersection of this submanifold with $\mathbb{R}_+^k{\times}\mathbb{R}^n$ 
is a manifold with corners.  The $k$ codimension one faces of this manifold are given by $${\partial}_iN{\hspace{3pt}}:=\hspace{3pt}N\hspace{3pt}{\cap}\hspace{3pt}(\mathbb{R}^k_+(a){\times}\mathbb{R}^n)$$ for $a=(1,...,1,0,1,...,1)$.  All of the higher codimension faces are of the same form but with other sequences $a$ (it is entirely possible that these faces may be empty).  Moreover the higher codimension strata fit together by inclusions induced by the inclusions of $\mathbb{R}^k_+(a){\times}\mathbb{R}^n$ into $\mathbb{R}^k_+(b){\times}\mathbb{R}^n$ when $a_i\;{\leq}\;b_i$ for all $1\;{\leq}\;i\;{\leq}\;k$.\\

In $\cite{Laures}$ some basic features of $\brac$-manifolds are observed: each of the submanifolds $M(a)$ is an $\langle{a}\rangle$-manifold and the product of an $\brac$-manifold with an $\langle{l}\rangle$-manifold is a $\langle{k+l}\rangle$-manifold. 

Now we give an intrinsic definition of $\brac$-manifolds as a special kind of manifolds with corners.  The basic definition of a manifold with corners can be found in chapter 14 of J. Lee's book on manifolds, $\cite{JLee}$.  Let $U$ and $V$ be open subsets of $\mathbb{R}^k_+$.  We say $f:U{\rightarrow}V$ is a diffeomorphism if $f$ extends to an honest diffeomorphism $F$ between two open sets in $U',V'{\subset}\mathbb{R}^k$ such that $U=U'\cap\mathbb{R}^k_+$ and $V=V'\cap\mathbb{R}^k_+$.  If $O$ is an open subset of a topological manifold with boundary, we say a pair $(O,\phi)$ is a chart with corners if $\phi$ is a homeomorphism from $O$ to an open set $U\;{\subset}\;\mathbb{R}^k_+$.  Two charts $(O,\phi)$, $(P,\psi)$ are compatible if $\psi{\circ}{\phi}^{-1}:{\phi}(O{\cap}P)\rightarrow\psi(O{\cap}P)$ is a diffeomorphism in the sense described above.
\begin{defn}
A smooth structure with corners on a topological manifold $M$ is a maximal collection of compatible charts with corners whose domains cover $M$.  A topological manifold together with a smooth structure with corners is called a smooth manifold with corners.
\end{defn}

As in $\cite{Laures}$, for $x=(x_1,...,x_k)\;{\in}\;\mathbb{R}^k_+$, $c(x)$ is defined to be the number of $x_i$'s that are zero.  For $m{\in}M$, $c(m)$ is defined as $c(\phi(x))$.  This number is independent of the chart $\phi$.  A $\emph{face}$ of $M$ is a union of connected components of the space $\{m{\in}M\;|\;c(m)=1\;\}$.
\begin{defn}A $\brac$-manifold is a manifold $M$ with corners together with an ordered k-tuple $(\partial{_1}M,...,\partial{_k}M)$ of subspaces satisfying the following properties:\\
i) each $\partial{_i}M$ is a closure of a face of $M$,\\
ii) each $m{\in}M$ belongs to $c(m)$ of the ${\partial}_iM$'s,\\
iii) for all $1{\leq}i{\neq}j{\leq}k$, $\partial{_i}M\;{\cap}\;\partial{_j}M$ is the closure of a face of $\partial{_i}M$ and $\partial{_j}M$.\end{defn}

Examples of $\brac$-manifolds are closed manifolds ($\langle{0}\rangle$-manifolds), manifolds with boundary ($\langle{1}\rangle$-manifolds) and the unit k-dimensional cube (a $\brac$-manifold).  $\brac$-manifolds have the feature that they admit the structure of a $\langle{k+1}\rangle$-manifold.  To do this, set $\partial{_{k+1}}M=\emptyset$.\\

For more interesting examples, consider the the closed unit 3-ball that has a pinch around the equator (a $\langle{2}\rangle$-manifold) and the tetrahedron which is a $\langle{4}\rangle$-manifold (it is also easy to check the tetrahedron does not admit the structure of a $\langle{3}\rangle$-manifold).  The prototypical $\brac$-manifold is $\mathbb{R}^k_+$.   The compactification of $\mathbb{R}^k_+$ is a space of fundamental interest to us.   To that end,  define
 $$(\mathbb{R}^k_+)^c:=\text{the one point compactification of }\mathbb{R}^{\infty}.$$
It is easy to see that $(\mathbb{R}^k_+)^c$ is a $\brac$-manifold and the inclusion map 
$$\mathbb{R}^k_+\longrightarrow(\mathbb{R}^k_+)^c$$

\noindent 
is a map of $\brac$-manifolds.  As a word of warning, it should be noted that 
ers do not admit the structure of a $\brac$-manifold.  A simple example is the cardioid.\\

\begin{figure}[h]
\begin{center}
\vspace{-30pt}
\includegraphics[bb = 0 0 235 180, width=0.2\textwidth]{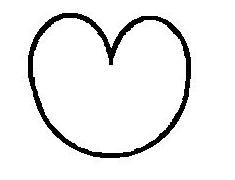}
\caption{The cardioid is a manifold with corners, but is not a $\brac$-manifold.}
\end{center}
\end{figure}

The reader has surely noticed the strata of a $\brac$-manifold and the inclusion maps between them form a (hyper)cubical diagram of spaces.  Such diagrams will be called $\brac$-spaces and now we give them proper definitions and highlight their basic features.\\

Let $\underline{2}^k$ be the poset whose objects are k-tuples of binary numbers.  Given two elements $a$ and $b\;{\in}\;\underline{2}^k$, $a$ is less than $b$ if when we write $a=(a_1,...,a_k)$ and $b=(b_1,...,b_k)$ we have $a_i\;{\leq}\;b_i$ for each $1\;{\leq}\;i\;{\leq}\;k$.  In this instance we shall write $a<b$.

\begin{defn}A $\brac$-space is a functor from $\underline{2}^k$ to $\emph{Top}$, the category of topological spaces.\end{defn}

\begin{defn}$\brac$-Top is the category whose objects are $\brac$-spaces, and whose morphisms are the natural transformations of these functors.\end{defn}

To obtain a $\brac$-space from a $\brac$-manifold $M$, set $M(($1,...,1$))$ equal to $M$ itself and for $a{\neq}1\;{\in}\;\underline{2}^k$ set 

$$M(a)\hspace{5pt}=\hspace{5pt}\underset{\{i|a_i=0\}}{\bigcap}{\partial}_iM.$$
The maps $M(a<b)$ are the inclusion maps.  We have just described a functor\\
$$\brac-\textit{manifolds}\;\;\longrightarrow\;\;\brac-Top$$
We continue on to describe some functors between the $\brac$-$\textit{Top}$ categories that are of fundamental importance. 
First, there is the product functor 
$$\underline{2}^{k+l}\;\;{\longrightarrow}\;\;\underline{2}^k{\times}\underline{2}^l\;\;{\longrightarrow}\;\;\textit{Top}{\times}\textit{Top}\;\;{\longrightarrow}\;\;\textit{Top}$$
which extends the product functor for $\brac$-manifolds.\\

For $1\;{\leq}\;i\;{\leq}\;k$ there are functors $\partial_{i},\overline{\partial_{i}}:\underline{2}^{k-1}\rightarrow\underline{2}^k$ that
insert a $0$ and $1$ respectively in between the ${i-1}^{st}$ and the $i^{th}$ slot of $a\,{\in}\,\underline{2}^{k-1}$.  Each of these functors induces
a functor
$$\partial_{i},\overline{\partial_{i}}:\brac-Top\;\;{\longrightarrow}\;\; \langle{k-1}\rangle-Top.$$
In the other direction, for $1{\leq}i{\leq}k$ there are functors
$\mathcal{I}_i:\underline{2}^k\rightarrow\underline{2}^{k-1}$ that forget the i$^{th}$ component of the vector $\twokob$.
Each of the functors $\mathcal{I}_i$ induces a functor $$\mathcal{I}_i:\langle{k-1}\rangle-Top\;\;{\longrightarrow}\;\;\brac-Top.$$
  Finally, there is a functor 
$$\emph{Top}\;{\longrightarrow}\;\brac-\textit{Top}$$ which sends a space $X$ to the constant $\brac$-space where every space $X(a)$ is equal to $X$ and every map $X(a<b)$ is the identity.\\

Next we give an analogue of homotopy in the $\brac$-$\textit{Top}$ category.  Let $f_0,f_1:X{\rightarrow}Y$ be $\brac$-maps.  Consider $[0,1]$ as a $\langle{0}\rangle$-space; thus $X{\times}I$ is also a $\brac$-space.  After identifying $X$ with $X{\times}\{0\}$ and $X{\times}\{1\}$ there are inclusion maps ${\iota}_0,{\iota}_1:X{\rightarrow}X{\times}I$ respectively.

\begin{defn}A homotopy between $\brac$-maps $f_0$ and $f_1$ is a $\brac$ map $F:X{\times}I{\rightarrow}Y$ such that $f_0={\iota}_0{\circ}F$ and $f_1={\iota}_1{\circ}F$.\end{defn}

We should mention there are a number of conventions we hold for elements of $\underline{2}^k$. Provided it leads to no unambiguity, we use $0$ as shorthand for $(0,...,0)$, and $1$ will stand for $(1,...,1)$.  We reserve $e_i$ to be the k-tuple whose components are all zero except for the $i$th. When $\twokmor$, we define 
$$b-a\;\;:=\;\;(b_1-a_1,...,b_k-a_k)$$
and for an element $\twokob$, $a'$ is defined to be $1$-$a$.   Lastly, $|a|$ denotes the number of non-zero component of $a$.\\

Now we discuss embeddings of $\brac$-manifolds.
\begin{defn} A neat embedding of a $\brac$-manifold $M$ is a natural transformation $\iota:M\;{\rightarrow}\;\mathbb{R}^{k}_+{\times}\mathbb{R}^N$ for some $N\;{\geq}\;0$ such that:\\
i) $\twokob$, $\iota(a)$ is an embedding,\\
ii) for all $\twokmor$, $M(b)\hspace{3pt}{\cap}\hspace{3pt}\mathbb{R}^k_+(a){\times}\mathbb{R}^N=M(a),$\\
iii) these intersections are perpendicular, i.e. there is some $\epsilon>0$ such that $$M(b)\hspace{3pt}{\cap}\hspace{3pt}(\mathbb{R}^k_+(a)\hspace{1pt}{\times}\hspace{1pt}[0,\epsilon]^k(b-a)\hspace{1pt}{\times}\hspace{1pt}\mathbb{R}^N)=M(a){\times}[0,\epsilon]^k(b-a).$$
We let $\text{Emb}(M,\mathbb{R}^k_+{\times}\mathbb{R}^N)$ denote the space of all neat embeddings.\end{defn}

In $\cite{Laures}$, Laures proved that every $\brac$-manifold neatly embeds in\\ $\mathbb{R}^{k}_+{\times}\mathbb{R}^N$ for some $N>>0$.  This shows that any $\brac$-manifold is $\brac$-diffeomorphic to one of the examples constructed in the beginning of the chapter.  Just like with closed manifolds, not only is the embedding space of a $\brac$-manifold nonempty, but for sufficiently large $N$ it is weakly contractible.  The main result of this section is

\begin{thm}\label{mainthm2}$\text{Emb}(M,\mathbb{R}^k_+{\times}\mathbb{R}^N)$ is weakly contractible provided $N$ is sufficiently large.\end{thm}

\begin{proof}The theorem follows from a proposition which we now describe.\\

Let $M$ be a $\brac$-manifold.  Let $A$ be a closed set in $M(1)$ and let $U$ be an open set in $M(1)$ containing $A$.  We obtain $\brac$-space structures on $A$ and $U$ by setting $A(a)$ equal to $A\cap{M}(a)$ and $U(a)$ equal to $U\cap{M}(a)$.  Moreover,  $U$ is a $\brac$-manifold.  Let $e:U\rightarrow\mathbb{R}^k_+{\times}\mathbb{R}^N$ be a neat $\brac$-embedding.  Finally, let us define $\text{Emb}(M,\mathbb{R}^k_+\times\mathbb{R}^N,e)$ to be the set of neat $\brac$-embeddings of $M$ that restrict to $e$ on $A$.

\begin{prop}\label{proptomain2}$\text{Emb}(M,\mathbb{R}^k_+{\times}\mathbb{R}^{N},e)$ is nonempty for $N$ sufficiently large.\end{prop}

$\textit{Proof of theorem }$\ref{mainthm2}$\textit{ from proposition }$\ref{proptomain2}. Let ${\phi}$ represent an element of ${\pi}_n(\text{Emb}(M,\mathbb{R}^k_+\times\mathbb{R}^N))$.  The product of the adjoint of $\phi$ with the standard embedding of $S^n$ in $\mathbb{R}^{n+1}$ yields an embedding of $\brac$-manifolds,
$$\overline{\phi}:M{\times}S^n\;{\longrightarrow}\;\mathbb{R}^k_+{\times}\mathbb{R}^{N+n+1}.$$ This map may be extended to a neat smooth embedding of $\langle{k+1}\rangle$-manifolds $$\overline{\phi}:[0,{\epsilon}){\times}M{\times}S^n\;{\longrightarrow}\;\mathbb{R}^{k+1}_+{\times}\mathbb{R}^{N+n+1}.$$  Now apply proposition \ref{proptomain2} to find a $\langle{k+1}\rangle$-embedding
$$\overline{\Phi}:M{\times}D^{n+1}\;\longrightarrow\;\mathbb{R}^{k+1}_+{\times}\mathbb{R}^{N'}$$
which extends $\overline{\phi}$.
The map $${\Phi}:D^{n+1}\rightarrow\text{Emb}(M,\mathbb{R}^k_+{\times}\mathbb{R}^{N'})$$ defined by ${\Phi}(d):=\overline{\Phi}(d,-)$ is a nullhomotopy of $\phi$.  Thus the class $[\phi]\;{\in}\;\text{Emb}(M,\mathbb{R}^k_+{\times}\mathbb{R}^{N'})$ is trivial. $\Box$\\

It remains to prove proposition \ref{proptomain2}.\\

\begin{proof}$\textit{ (adapted from }\cite{Laures}\textit{)}$.  A halfway marker in the proof is

\begin{lemma}\label{extendtoneighborhood}  Let $\twokob$, let $A$ be a closed set in $M(1)$, and $U$ some open neighborhood of $A$ in $M(1)$ together with an embedding of $\brac$-spaces
$$e:U{\cup}M(a)\;\longrightarrow\;\mathbb{R}^k_+{\times}\mathbb{R}^{N}.$$
which restricted to $U$ is a $\brac$-embedding and restricted to $M(a)$ is a $\langle{|a|}\rangle$-embedding.  There is some neighborhood $\tilde{U}$ of $A{\cup}M(a)$ in $M(1)$ and a $\brac$-embedding
$$\tilde{e}:\tilde{U}\;\longrightarrow\;\mathbb{R}^k_+{\times}\mathbb{R}^N$$
which agrees with $e$ on the intersection of $\tilde{U}$ and $U{\cup}M(a)$.
\end{lemma}

 Assuming this lemma for the moment we prove proposition \ref{proptomain2} as follows. Embed $M(0)$ in $\mathbb{R}^N$ extending the embedding $e$ on a neighborhood of $A$ in $M(0)$ using the relative Whitney Embedding theorem. Now apply lemma \ref{extendtoneighborhood} to find a $\brac$-embedding of a collar of $M(0)$ in $M(1)$ which agrees with $e$ on an open neighborhood of $A$ in $M(1)$. Now by induction suppose $\mathcal{B}\subset\underline{2}^k$ is a non-empty subset and suppose we have found a $\brac$-embedding, $e_0$ on an open neighborhood of $A\;{\bigcup}\;({\cup}_{b{\in}\mathcal{B}}M(b))$ in $M(1)$ that agrees with $e$ on an open neighborhood of $A$ in $M(1)$.  Call this neighborhood $N$, and now take $a\;{\in}\;\underline{2}^k-{\mathcal{B}}$ where without loss of generality we may assume that $\mathcal{B}$ contains all $b$ strictly less than $a$.  Pick a finite partition of unity for $U$, 
$$\{(u_j:U_j\;{\rightarrow}\;\mathbb{R}^{|a|}, {\Phi}_j)\}\;{\cup}\;(e_0:N{\cap}M(a)\;{\rightarrow}\;\mathbb{R}^k_+{\times}\mathbb{R}^N,{\Xi})$$ with the property that ${\cup}U_j$ is disjoint from $A\;{\bigcup}\;({\cup}_{b{\in}\mathcal{B}}M(b))$.  If we define $$e_1:M(a)\;{\rightarrow}\;\mathbb{R}^k_+{\times}\mathbb{R}^{N'}$$ by the formula $$({\Xi}e_0,{\Phi}_1u_1,...,{\Phi}_lu_l),$$ then $e_1$ is an embedding which agrees with $e$ on a (possibly smaller) neighborhood of $A\;{\bigcup}\;({\cup}_{b{\in}\mathcal{B}}M(b))$ in $M(a)$.  It also satisfies the hypothesis of lemma \ref{extendtoneighborhood}, setting this (possibly smaller) neighborhood equal to $U$ in the notation of the lemma.   But now by lemma \ref{extendtoneighborhood}
 we may find a $\brac$-embedding of an open neighborhood of $$A\;{\bigcup}\;(\underset{b{\in}\mathcal{B}{\cup}\{a\}}{\bigcup}M(b))$$ in $M(1)$ that restricts to $e$ on an open neighborhood of $A$ in $M(1)$.  Thus we may inductively extend $e$ to a $\brac$-embedding of $M$.

Now we prove lemma \ref{extendtoneighborhood}.\\

Put a Riemannian metric on $M(1)$.   Consider the section of the bundle $TM(1)$ restricted to $M(e_i')$ which we obtain by taking the inward pointing unit vector normal to $TM(e_i')$.  Use the geodesic flow generated by this vector field to find a $\brac$-embedding $M(e_i'){\times}[0,{\epsilon}){\rightarrow}M(1)$.  Differentiating the flow produces a vector field in a neighborhood of $M(e_i')$ in $M(1)$.  Let $V_i$ denote the pushfoward under $e$ of this vector field restricted to some neighborhood of $e(U(e_i'))$ in $e(U)$.  Let $E_i$ denote the standard constant unit vector field orthogonal to $\mathbb{R}^k_+(e_i'){\times}\mathbb{R}^N$. Using a partition of unity $$u_1{\cup}u_2=[0,\epsilon){\times}\mathbb{R}^k_+(e_i'){\times}\mathbb{R}^N$$ with $u_2$ disjoint from some neighborhood of $e(U(e_i'))$, we may extend $V_i$ to a neighbhorhood of $\mathbb{R}^k_+(e_i'){\times}\mathbb{R}^N$ in $\mathbb{R}^k_+{\times}\mathbb{R}^N$. This extension is transverse to $\mathbb{R}^k_+(e_i'){\times}\mathbb{R}^N$.  By construction, away from a neighborhood of $e(U(e_i'))$ the vector field agrees with $E_i$, and on a smaller neighborhood of $e(A(e_i'))$, it agrees with $V_i$.  This extended vector field $V_i$ generates a flow which for small $\epsilon$ produces an embedding $$F_i:[0,\epsilon){\times}\mathbb{R}^k_+(e_i'){\times}\mathbb{R}^N\;\longrightarrow\;\mathbb{R}^k_+{\times}\mathbb{R}^N.$$\\

  Now we are ready to define the embedding of a neighborhood of $M(a){\cup}A$ in $M(1)$.  Using the same Riemannian metric as before, we may identify a neighborhood of $M(a)$ in $M(1)$ with $M(a){\times}[0,\epsilon)^k(a')$, and more generally, a neighborhood of $M(a){\cup}A$ in $M(1)$ by $M(a){\times}[0,\epsilon)^k(a')\;{\cup}\;U$.  By picking $\epsilon$ smaller and picking a smaller $U$ (that still contains $A$) we may additionally assume that $$U\;{\cap}\;M(a){\times}[0,\epsilon)^k(a')\;=\;U(a){\times}[0,\epsilon)^k(a').$$
Define
$$\tilde{e}:M(a){\times}[0,\epsilon)^k(a')\;{\cup}\;U\;\longrightarrow\;\mathbb{R}^k_+{\times}\mathbb{R}^N$$
$$x=(m,t_1,...,t_{|a'|})\mapsto F_{|a'|}(t_{|a'|},{\cdots}F_1(t_1,e(m))\text{ if }x\;{\in}\;M(a){\times}[0,\epsilon)^k(a')$$
$$x{\mapsto}e(m)\text{ if }x\;{\in}\;U$$
By construction this map is well defined and is a $\brac$-embedding which extends the original embedding $e$ in some neighborhood of $A{\cup}M(a)$.  This concludes the proof of lemma \ref{extendtoneighborhood}\end{proof}
and hence also theorem \ref{mainthm2}.\end{proof}

\section{$\brac$-vector bundles and $\brac$-spectra}

$\indent$In this section we explore vector bundles: (stable) tangent bundles, normal bundles, and structures over these bundles for $\brac$-manifolds.  After describing this, we construct Thom spectra in the $\brac$-\textit{Top} category analogous to ordinary Thom spectra.  To keep indexing under control we will always use $N$ to stand for $d+n+k$.\\ 

\begin{defn}A $\brac$-vector bundle is a functor from $\underline{2}^k$ to the category of vector bundles.\end{defn}

The $\brac$-vector bundles that appear in our paper belong to a special class of $\brac$-vector bundles.

\begin{defn} A $\brac$-vector bundle $E$ is geometric if for every $\twokmor$
the vector bundle $E(b)$ restricted to the base of $E(a)$ splits canonically as  ${\epsilon}^{c}{\oplus}E(a)$ for some $c\;{\geq}\;0$.\end{defn}

The normal and tangent bundles of $\brac$-manifolds are geometric $\brac$-vector bundles.  As is the case with ordinary vector bundles, there is a universal example of geometric $\brac$-vector bundles which we now describe.

For the next definition, consider $\mathbb{R}$ as a $\langle{1}\rangle$-space by letting
$\mathbb{R}(0)=\{0\}$ and the map $0<1$ be the inclusion of $\{0\}$ into the real line.
\begin{defn}The $\brac$-Grassmanian, which we denote by $G(d,n)\brac$, is the $\brac$-space defined as follows.  For $\twokob$ let $G(d,n)(a)$ be the space of $d+|a|$ dimensional vector subspaces of $\mathbb{R}^k(a){\times}\mathbb{R}^{d+n}$.  For $\twokmor$ the map
$$G(d,n)(a)\;\;\;{\longrightarrow}\;\;\;G(d,n)(b)$$
is defined by 
$$S\,\,\,\mapsto\,\,\,S+\mathbb{R}^k(b\text{-}a).$$\end{defn}

The inclusion $\mathbb{R}^k{\times}\mathbb{R}^{d+n}\rightarrow\mathbb{R}^k{\times}\mathbb{R}^{d+n+1}$
induces a $\brac$ map\\
$${\sigma}_N:G(d,n)\brac\;\;{\longrightarrow}\;\;G(d,n\text{+}1)\brac.$$
A neatly embedded d+k-dimensional $\brac$-submanifold of $\mathbb{R}^k_+{\times}\mathbb{R}^{d+n}$ admits a map to G(d,n)$\brac$ by the rule \\
$$p\;\;{\mapsto}\;\;T_pM\;{\subset}\;\mathbb{R}^k(a){\times}\mathbb{R}^{d+n}.$$
We shall denote this $\brac$-map by\\
$${\tau}M_N:M\;\;{\longrightarrow}\;\;G(d,n)\brac$$
and in the limit
$${\tau}M:M\;\;{\longrightarrow}\;\;BO(d)\brac.$$

\begin{defn} The canonical geometric $\brac$-bundle which we denote by ${\geocanon}$ is a geometric $\brac$-vector bundle over $G(d,n)\brac$ defined as follows.  For $\twokob$, the total space of ${\gamma}(d,n)(a)$ is
$$\{\;(S,v)\;{\in}\;G(d,n)(a){\times}\mathbb{R}^k(a){\times}\mathbb{R}^{d+n}\;\;|\;\;v{\in}S\;\}.$$
${\gamma}(d,n)(a)$ is the bundle map sending $(S,v)$ to $S$.  For $\twokmor$ the map ${\gamma}(d,n)(a<b)$ sends
$$(S,v)\;\;\text{ to }\;\;(\;S+\mathbb{R}^k(b-a)\;,\;\mathbb{R}^k(a<b)(v)\;).$$
\end{defn}

The reader is encouraged to notice that for every $\twokob$, the bundle ${\gamma}(d,n)(a)$ is canonically isomorphic to the usual canonical bundle $\gamma(d\text{+}|a|,n)$.

\begin{defn} The canonical perpendicular geometric $\brac$-bundle which we denote by $\geoperpcanon$ is a geometric $\brac$-vector bundle over $G(d,n)\brac$ defined as follows.  For $\twokob$, the total space of ${\gamma}(d,n)(a)$ is
$$\{\;(S,v)\;{\in}\;G(d,n)(a){\times}\mathbb{R}^k(a){\times}\mathbb{R}^{d+n}\;\;|\;\;v{\in}S^{\perp}\;\}.$$
${\gamma}^{\perp}(d,n)(a)$ is the bundle map sending $(S,v)$ to $S$.  For $\twokmor$ the map ${\gamma}^{\perp}(d,n)(a<b)$ sends
$$(S,v)\;\;\text{ to }\;\;(\;S+\mathbb{R}^k(b-a)\;,\;\mathbb{R}^k(a<b)(v)\;).$$
\end{defn}

The reader is again encouraged to notice that for every $\twokob$, the bundle ${\gamma}^{\perp}(d,n)(a)$ is canonically isomorphic to the usual canonical perpendicular bundle ${\gamma}^{\perp}(d\text{+}|a|,n)$.\\

At this juncture, we remark that the pullback of a vector bundle is readily defined for $\brac$-bundles. 
 While the Thom space of a vector bundle doesn't generalize to all $\brac$-vector bundles,  it does make sense for a large class of them.

\begin{defn} Let $\nu:E\brac\;{\longrightarrow}\;B\brac$ be a $\brac$-vector bundle such that every vector bundle map $E(a<b)$ restricts to an injective map on every fiber.   Let $D(E)\brac$ and $S(E)\brac$ be the disk and sphere $\brac$-bundles associated to $\nu$.  The Thom space of $\nu$ is the $\brac$-space denoted $B^{\nu}$ where for each $\twokob$, $B^{\nu}(a)$ is the space
 $$\frac{D(E)(a)}{S(E)(a)}$$
and the map 
$$B^{\nu}(a)\;\;\longrightarrow\;\;B^{\nu}(b)$$
is induced by the map $E(a<b)$.
\end{defn}
Notice that the class of $\brac$-vector bundles which admit Thom spaces includes geometric $\langle{k}\rangle$-vector bundles.\\

The next couple of remarks are more particular to our study of cobordisms accomodating stable tangential structure.  In the early days of cobordism theory ($\cite{Sto}$) it was realized that Pontrjagin-Thom construction could be generalized to yield a bijection between the set of cobordism classes of manifolds together with some stable structure and a homotopy group of the Thom space of the canonical perpendicular bundle over $BO$ pulled back along a certain fibration.   This fibration over $BO$ corresponds to the extra stable tangential data.  We briefly lay the groundwork for what the analogue of these fibrations will be in the $\brac$-space setting.

\begin{defn} We say $\fullstruc$ is a structure if for each $n{\geq}0$ we are given a $\langle{k}\rangle$-space $\struc$ together with maps of $\langle{k}\rangle$-spaces
$$\strucmapN:\struc\;\;{\longrightarrow}\;\;G(d,n)\brac\text{ and}$$
$${\sigma}'_{N}\brac:\struc\;\;{\longrightarrow}\;\;B_{d+n+1}\brac$$
such that $\strucmapN(a)$ is a fibration for each $\twokob$ and $\strucmapNone{\circ}{\sigma}'{_N}=\sigma{_N}{\circ}\strucmapN$.\end{defn}

\begin{defn} Suppose $M$ is a neat $\brac$-submanifold of $\mathbb{R}^k_+{\times}\mathbb{R}^{\infty}$.  A $\fullstruc$-structure on $M$ is a sequence of $\brac$-maps $f_N:M{\longrightarrow}\struc$ for all $N$ greater than some natural number $c$ such that 
${\strucmapN}f_N={\tau}M_N$ and ${\sigma}_Nf_N=f_{N+1}$.\end{defn}

We close out this section by discussing $\brac$-prespectra.  The Thom spaces of the vector bundles sketched earlier will furnish us with important examples.  For the rest of this section we will work solely with pointed spaces and pointed maps.\\

\begin{defn} A $\brac$-(pre)spectrum is a functor from $\underline{2}^k$ to the category of (pre)spectra.  Likewise, a $\brac-\Omega$-spectrum is a functor from $\underline{2}^k$ to the category of $\Omega$-spectra.\end{defn}

\begin{example} Start with a $\brac$-space $X$.  We obtain a $\brac$-prespectrum ${\Sigma}^{\infty}X$ by setting $({\Sigma}^{\infty}X)(a)$ to the prespectrum ${\Sigma}^{\infty}(X(a))$.\end{example}

\begin{example} Let $X_N$ be the $\brac$-space $(\mathbb{R}^k_+{\times}\mathbb{R}^{N-k})^c$ and $\sigma{_N}$ the obvious $\brac$-isomorphism.  Alternatively, we could have defined this spectrum kth desuspension of ${\Sigma}^{\infty}(\mathbb{R}^k_+)^c$.  It is the analogue of the sphere spectrum in the $\brac$-spectrum setting.\end{example}

\begin{defn} Let the Thom spectrum $\plainthom\brac$ be the $\brac$-prespectrum defined as follows.  For $\twokob$, let $\plainthom(a)$ be the Thom spectrum whose $N^{th}$ space is the Thom space of ${\gamma}^{\perp}(d,n)(a)$.  The reader may notice this spectrum is isomorphic to the Thom spectrum ${\Sigma}^{|a|-k}MT(d+|a|)$.   For $\twokmor$, let $\plainthom(a<b)$ be the degree zero map of prespectra which on the $N^{th}$ space is simply the map of Thom spaces induced by the bundle map.
$${\gamma}^{\perp}(d,n)(a)\;\;\longrightarrow\;\;{\gamma}^{\perp}(d,n)(b).$$
\end{defn}

As an elaboration on the example above, suppose $\fullstruc$ is a structure.
\begin{defn} We define the $\brac$-prespectrum $\thom\brac$ as follows.  For $\twokob$, let $\thom(a)$ be the Thom spectrum whose $N^{th}$ space is the Thom space of ${\theta}^*{\gamma}^{\perp}(d,n)(a)$.  For $\twokmor$, let $\plainthom(a<b)$ be the degree zero map of prespectra which on the $N^{th}$ space is simply the map of Thom spaces induced by the bundle map.
$${\theta}^*{\gamma}^{\perp}(d,n)(a)\;\;\longrightarrow\;\;{\theta}^*{\gamma}^{\perp}(d,n)(b).$$
\end{defn}

The indices can get confusing.  For this reason, the reader may find it easier to view $\brac$-prespectra as a sequence of $\brac$-spaces together with maps of the suspension of each $\brac$-space to the following $\brac$-space.  In the first of the two preceding definitions, 
$$\text{the }N^{th}\text{ $\brac$-space of }\plainthom\text{ is }\plainthomN\brac$$
and in the second definition,
$$\text{the }N^{th}\text{ $\brac$-space of }\thom\text{ is }\thomN\brac$$
It is in order to note that the preceding definition is an example of a $\brac$-diagram of Thom spectra associated to virtual bundles.

Next we describe a functor from $\brac$-prespectra to ordinary (not $\brac$!) $\Omega$-spectra.

\begin{defn}For any natural number $n$ and pointed $\brac$-space $X$, let ${\Omega}^n_{\brac}X$ be the space
 of based $\brac$-maps from $(\mathbb{R}^k_+{\times}\mathbb{R}^{n-k})^c$ to $X$.
\end{defn}

Given a $\brac$-prespectrum $\underline{X}=\{X_N,{\sigma}_N\}$ we may form the $\Omega$-spectrum $\underline{{\Omega}^{\infty}_{\brac}X}$ whose $L$th space is given by $$\underset{N\rightarrow\infty}{\text{colim}}\;{\Omega}^{N-L}_{\brac}X_N.$$

Soon we will see that the zero spaces of $\Omega$-spectra arising in this fashion are geometrically significant.  Because of their centrality in the rest of the paper, we will abrreviate the zero space functor from $\brac$-prespectra to infinite loopspaces by just ${\Omega}^{\infty}_{\brac}$.
An important proposition of Laures describes the homotopy type of ${\Omega}^{\infty}_{\brac}\underline{X}$ as being the zero space of an ordinary prespectrum.  The spaces in the prespectrum are given by taking a homotopy colimit of an appropriate functor over the category $\underline{2}^k_*$ (notice the extra asterisk) which we now describe.

Following $\cite{Laures}$, let $\underline{2}^k_*$ be the category consisting of $\underline{2}^k$ plus an extra object, $*$, and one morphism from each object to $*$ \textit{with the exception} of $1\;{\in}\;{2}^k$.  If $X$ is a functor from $\underline{2}^k$ to the category of pointed spaces, then $X_*$ is the functor from $\underline{2}^k_*$ to pointed spaces defined by $X_*(*)=\{\text{ basepoint}\}$, and otherwise $X_*(a):=X(a)$ for all $\twokob$.\\

Now if $\underline{X}=\{X_N,{\sigma}_N\}$ is a $\brac$-prespectrum, the spaces $\underset{\underline{2}^k_*}{\text{hocolim}{X_N}_*}$ form a prespectrum as there are maps
$$\Sigma\;\underset{\underline{2}^k_*}{\text{hocolim}}{{X_N}_*}\;{\approxeq}\;\underset{\underline{2}^k_*}{\text{hocolim}}{{\Sigma}X_N}_*\;\longrightarrow\;\underset{\underline{2}^k_*}{\text{hocolim}}{{X_{N+1}}_*}$$
We denote this prespectrum by $\underset{\underline{2}^k_*}{\text{hocolim}}\underline{X}_*$.\\

\begin{prop}\label{PuppeLaures}($\cite{Laures}$, lemmas 3.1.2 and 3.1.4)  Let $\underline{X}=\{X_N,{\sigma}_N\}$ be a $\brac$-prespectrum.  Then
$$\Omega^{\infty}_{\brac}\underline{X}\;\;{\simeq}\;\;
{\Omega}^{\infty}\underset{\underline{2}^k_*}{\text{hocolim}}\underline{X}_* $$\end{prop}

At this point, we can state the analogue of the usual Pontrjagin-Thom theorem for $\brac$-manifolds.

\begin{thm}($\cite{Laures}$) The set of cobordism classes of d-1+k dimensional $\brac$-manifolds is in
 natural bijection with $${\pi}_0\;{\Omega}^{\infty}\underset{\underline{2}^k_*}{\text{hocolim}}\;G^{-{\gamma}_{d-1}}\brac_*.$$
\end{thm}
More generally if we are given a structure $\fullstruc$ there is also the following theorem. 
\begin{thm}($\cite{Laures}$) The set of cobordism classes of d-1+k dimensional $\brac$-manifolds together with lifts to $\fullstruc$ is in natural bijection with $${\pi}_0\;{\Omega}^{\infty}\underset{\underline{2}^k_*}{\text{hocolim}}\;B^{-{\theta}^*{\gamma}_{d-1}}\brac_*.$$
\end{thm}

\begin{proof}
We prove the second more general statement.  By the proposition above it suffices to show that there is a natural correspondance between cobordism classes and the set
$${\pi}_0\;{\Omega}^{\infty}_{\brac}\;B^{-{\theta}^*{\gamma}_{d-1}}\brac.$$
Let $(M,f)$ be a neat $\brac$-submanifold of $\mathbb{R}^k_+{\times}\mathbb{R}^{\infty}$ which representing the cobordism class $[(M,f)]$.  As $M$ is compact it lies in $\mathbb{R}^k+{\times}\mathbb{R}^{d-1+n}$ for some $n{\geq}0$.  For each $\twokob$ we may take a Thom collapse map
$$C(a):\;(\mathbb{R}^k_+(a){\times}\mathbb{R}^{d-1+n})^c\;\;\longrightarrow\;\;(Tub(a))^c$$
onto the one-point compactification of a tubular neighborhood of $M(a)$ in $\mathbb{R}^k_+(a){\times}\mathbb{R}^{d-1+n}$.  By the tubular neighborhood theorem there is some homeomorphism $t(a)$, between ${Tub}(a)$ and the pullback of ${\gamma}^{\perp}(d-1,n)(a)$ to $M(a)$ by ${\tau}M(a)$.  Let $\pi$ be the projection of $Tub(a)$ onto $M$.  Then $(\;f(a){\circ}{\pi}\;,\;t(a)\;)$
defines a proper map
$$Tub(a)\;\;\longrightarrow\;\;{\theta}^*_{N-1}{\gamma}^{\perp}(d-1,n)(a).$$
Composing $C(a)$ with the compactification of the above map yields a map
$$h(a):\;(\mathbb{R}^k_+(a){\times}\mathbb{R}^{d-1+n})^c\;\;{\longrightarrow}\;\;B^{{\theta}^*_{N-1}{\gamma}^{\perp}(d-1,n)}(a).$$
The $h(a)$ maps cobble together to yield a map $H_{N-1}$ of $\brac$-spaces:
$$H_{N-1}\;{\in}\;{\Omega}^{N-1}_{\brac}B^{{\theta}^*_{N-1}{\gamma}^{\perp}(d-1,n)}\brac.$$
This map is well defined up to $\brac$-homotopy.\\

To construct an inverse map we start with a map
$$H\;{\in}\;{\Omega}^{\infty}_{\brac}B^{-{\theta}^*{\gamma}^{\perp}_{d-1}}\brac$$
 which induces for some $N$, a $\brac$-map $$H_{N-1}\;{\in}\;{\Omega}^{N-1}_{\brac}B^{{\theta}^*_{N-1}{\gamma}^{\perp}(d-1,n)}\brac.$$
  By Thom transversality for $\brac$-maps (see appendix \ref{appendix2}) we may perturb  ${\theta}_{N-1}{\circ}H_{N-1}$ by some $\brac$ homotopy $l$ to a map $\overline{H}_{N-1}$ that is transverse to the submanifold $G(d-1,n)\brac$ of ${\gamma}^{\perp}(d-1,n)\brac$.  The preimage $$M:=\;\overline{H}_{N-1}^{-1}(G(d-1,n)\brac)$$ 
will be a $d-1+k$ dimensional $\brac$-submanifold of $\mathbb{R}^k_+{\times}\mathbb{R}^{d-1+n}$; notice the map $\overline{H}_{N-1}$ restricted to $M$ is simply the classifying map of its tangent bundle, ${\tau}S$, which we will now use as notation.  Furthermore, we may also choose our homotopy $l$ so that ${\tau}M$ is arbitrarily $C_0$ close to ${\theta}_{N-1}{\circ}H_{N-1}$.  Because of this, we may assume that $M$ is a $\brac$-subspace of $H_{N-1}^{-1}{\theta}_{N-1}^{-1}\gamma(d-1,n)\brac$ or equivalently that the map $H_{N-1}$ restricted to $M$ maps to $B(d-1,n)\brac$.  As ${\theta}_{N-1}$ is a fibration, the data of $H_{N-1}$ together with the homotopy $l$ between ${\theta}_{N-1}{\circ}H_{N-1}$ and ${\tau}S$ furnish a map $$f:M\;\longrightarrow\;B(d-1,n)\brac$$ such that ${\theta}_{N-1}{\circ}f={\tau}M$.  This yields a pair $[(M,f)]$ which is the desired cobordism class.  The two constructions are inverse to one another.
\end{proof}

\section{$\brac$-Cobordism Categories and Statement of Main Theorem}

$\indent$The first goal of this section is to define cobordism categories for $\brac$-manifolds.  The definition will be broad enough to include tangential data.  Once we have done this, we state the main theorem of the paper.\\

Our cobordism categories are topological categories.  We will first describe the category as a category of sets, and later indicate how to topologize the category.  For the remainder of this section, let $\fullstruc$ be a structure over $G(d,n)\brac$.  Also recall that for a $\brac$-submanifold $M$ of $\mathbb{R}^k_+{\times}\mathbb{R}^{\infty}$, ${\tau}M:\;M\;{\longrightarrow}\;BO(d)\brac$ is the map classifying the tangent bundle of $M$.
\begin{defn}Let $\Cob$ denote the d+$\brac$ topological cobordism category with structure $\fullstruc$.\\ 
$\textbf{Objects of }\Cob\textbf{ :}$\\
As a set, the objects of $\Cob$ consist of triples $(N,a,g)$ where:\\
  i) $a\;{\in}\;\mathbb{R}$,\\
 ii) $N$ is a neat d-1+k dimensional $\brac$-submanifold of $\mathbb{R}^k_+{\times}\mathbb{R}^{\infty}{\times}\{a\}$,\\
 iii) $g:N\;{\longrightarrow}\;\underset{n\rightarrow\infty}{\text{colim}}\;\struc$ satisfies ${\strucmap}{\circ}g={\tau}N'_{|N}$ where $$N'=\{\;(\vec{x},\vec{y},z)\;{\in}\;\mathbb{R}^k_+{\times}\mathbb{R}^{\infty}{\times}\mathbb{R}\;\;|\;\;(\vec{x},\vec{y},a)\;{\in}\;N\;\}.$$\\
$\textbf{Morphisms of }\Cob\textbf{ :}$\\
The set of morphisms from $(N_1,a,g_1)$ to $(N_2,b,g_2)$ consist of quadruples $(W,a,b,G)$ where:\\
 i) $W$ is a neat $d+k$-dimensional $<$k+1$>$-submanifold of $
\mathbb{R}^k_+{\times}\mathbb{R}^{\infty}{\times}[a,b]$,\\
 ii) $G:W\;{\longrightarrow}\;\underset{n\rightarrow\infty}{\text{colim}}\;\mathcal{I}_{k+1}\struc$ satisfies $\strucmap{\circ}G={\tau}W$,\\
 iii) ${\partial}_{\text{k+1}}W\;=\;N_1\;{\coprod}\;N_2$ and $\partial_{\text{k+1}}G\;=\;g_1\;\coprod\;{g_2}$\\
 (for a $\brac$-map $f:X{\rightarrow}Y$, ${\partial}_if$ stands for $f{_{|{\partial}_iX}}$).\end{defn}

Here is an example when $\theta$ is trivial and $d=k=1$.  Notice that $W(1,0)$ is the disjoint union of figures $(b)$ and $(c)$.  Figure $(e)$ could also have been labelled $N_1(0){\coprod}N_2(1)$.

\begin{figure}[H]
\vspace{-50 pt}
\centering
\hspace{40 pt}
\subfloat[$W(1,1)$]{\includegraphics[bb = 0 0 581 609, width=0.60\textwidth]{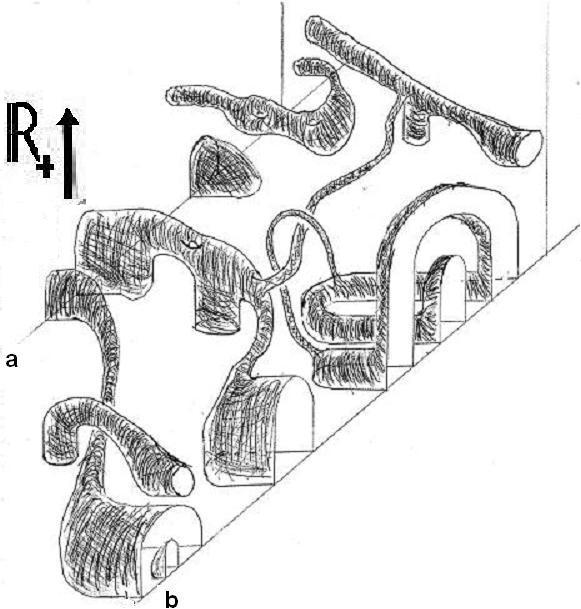}}
\\
\centering
\vspace{-25 pt}
\subfloat[$N_1(1)$]{\includegraphics[bb = 0 0 457 430, width=0.24\textwidth]{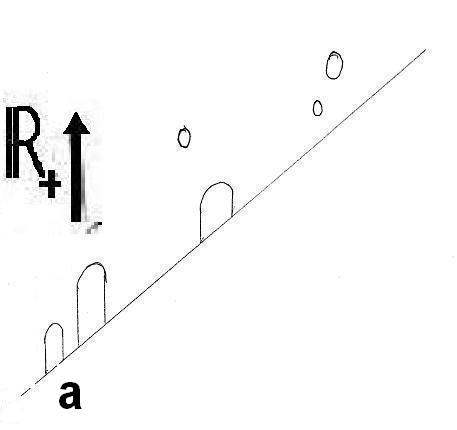}}
\subfloat[$N_2(1)$]{\includegraphics[bb = 0 0 544 527, width=0.24\textwidth]{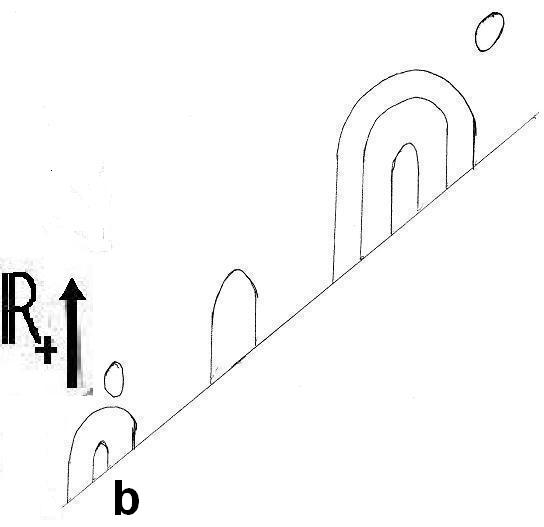}}
\subfloat[$W(0,1)$]{\includegraphics[bb = 0 0 613 593, width=0.24\textwidth]{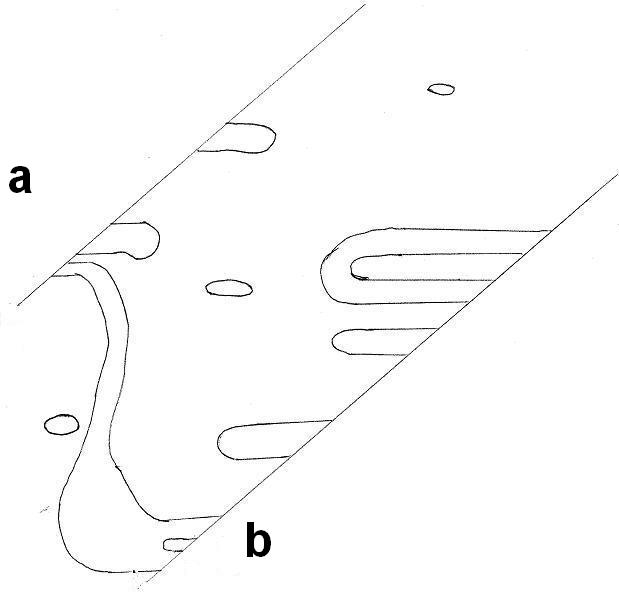}}
\subfloat[$W(0,0)$]{\includegraphics[bb = 0 0 588 631, width=0.24\textwidth]{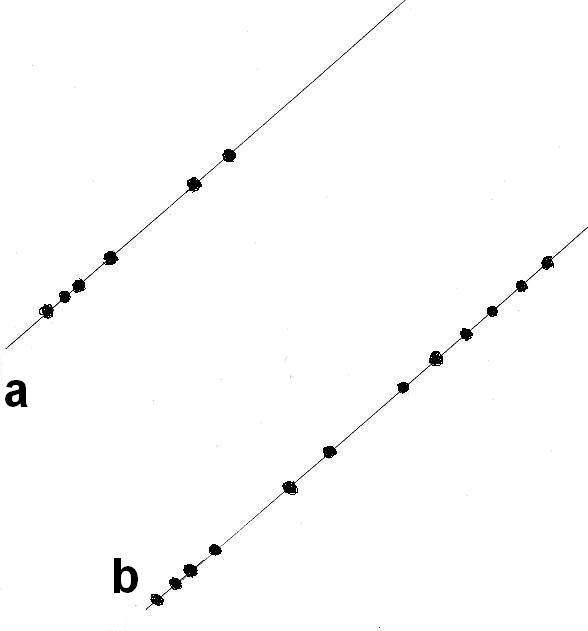}} 
\end{figure}

 Now we put a topology on the set of objects.  As a preliminary step, notice that the space of $\langle{k}\rangle$-maps, $C_{\brac}(X,Y)$, can be written as an appropriate limit over a diagram where the spaces are of the form $C(X(a),Y(b))$ for $a,b{\in}\underline{2}^k$ and $C(X(a),Y(b))$ are endowed with
 the compact open topology.\\

\begin{defn}For a d-1+k dimensional $\brac$-manifold $M$, let $Y(M)_n$ denote the space of neat $\brac$-embeddings of $M$ into $\mathbb{R}^k_+{\times}\mathbb{R}^{d+n}$ together with a structure map from $M$ to $\struc$.\end{defn}
 By definition then, $Y(M)_n$ is the limit of the following diagram:
 $$\text{Emb}_{\brac}(M,\mathbb{R}^k_+\times\mathbb{R}^{d+n}){\rightarrow}C_{\brac}(M,G(d,n,\brac)){\leftarrow}C_{\brac}(M,\struc))$$
 which is topologized as a subset of
$$\text{Emb}_{\brac}(M,\mathbb{R}^k_+\times\mathbb{R}^{d+n})\;\;{\times}\;\;C_{\brac}(M,\struc))$$
\begin{defn}\label{defncobobs} Let $Y(M):=\text{colim}_{n\rightarrow\infty}Y(M)_n$.\end{defn}
\begin{defn} Let $\text{Diff}_{\brac}(M)$ denote the space of diffeomorphisms of $M$ that map $M(a)$ to $M(a)$ for each $\twokob$.\end{defn}
 Diff$_{\brac}(M)$ clearly acts on $Y(M)$, and the space 
$\frac{Y(M)}{\text{Diff}_{\brac}(M)}$ comes equipped with the quotient topology.  Then
$$\text{ob}\;\Cob\;\;:=\;\;\underset{[M]{\in}\mathcal{S}}{\coprod}\;\frac{Y(M)}{\text{Diff}_{\brac}(M)}{\times}\mathbb{R}$$
where $\mathcal{S}$ is the set of diffeomorphism classes of d-1+k dimensional $\brac$-manifolds topologizes the objects of $Cob$\\

 We topologize the morphism space as we did the object space.  Composition is cobordism composition and gluing the maps $G$ along their shared domain.\\

Now we are ready to state the main theorem of the paper.  Again, recall that  $\fullstruc$ is a structure over $G(d,n)\brac$.
\begin{mainthm}\label{mainthm}
$$B\Cob\;\;{\simeq}\;\;\kzerospec$$
\end{mainthm}
The proof of this theorem will be the subject of the next section.    Combining the theorem with the lemma from the previous section we obtain an important corollary which we will use for our applications.
\begin{maincor}$$B\Cob\;\;{\simeq}\;\;{\kzerohocolim}$$\end{maincor}

\section{Proof Of Main Theorem}

$\indent$The strategy of proof is to adapt the argument found in ($\cite{GMTW}$) to the setting of $\brac$-manifolds.  We give an overview now; it should be noted that the original insights all come from $\cite{M-W}$.   First, there is a natural bijection between homotopy classes of maps, $[X,\kzerospec]$, and the $\textit{concordance}$ classes of an associated sheaf of sets on $X$ (our sheaves here are defined on manifolds without boundary).  We describe this associated sheaf and explain what we mean by concordance.  The proof of the bijection is postponed to the end of the section.

\begin{defn}Let $\fullstruc$ be tangential data, let $U$ be a manifold without boundary, and let $N=d+n+k$. We define  $\sheafDn(U)$ to be the set of pairs $(W,g)$ where $W$ is a neat $\brac$-submanifold of $U\hspace{3 pt}{\times}\hspace{3 pt}\mathbb{R}\hspace{3 pt}{\times}\hspace{3 pt}\mathbb{R}^k_+\hspace{3 pt}{\times}\hspace{3 pt}\mathbb{R}^{d-1+n}$ and $g:W\;{\longrightarrow}\;\struc$ is a $\brac$-map which satisfy the criteria listed below.  For what follows, let $\pi$ and $f$ be the projection maps onto $U$ and $\mathbb{R}$ respectively.  Then\\
\newline
i) $(\pi,f)$ is proper.\\
\newline
ii) $\pi$ is a submersion whose fibres are d+k-dimensional neat $\brac$-submanifolds of $\mathbb{R}\;{\times}\;\mathbb{R}^{k}_+\;{\times}\;\mathbb{R}^{d-1+n}.$\\
\newline
iii) $\strucmapN{\circ}g=T^{\pi}W_N$, where $T^{\pi}W_N$ is the map sending $x\;{\in}\;W$ to the tangent space at $x$ of the manifold $\pi^{-1}(\pi(x))$ thought of as a subspace of $\{x\}\times\mathbb{R}^{\brac}{\times}\mathbb{R}^{d+n}$ i.e. an element of $G(d,n)\brac$.\end{defn}

\begin{defn}We define $\sheafD(U):=\text{colim}_{n{\rightarrow}\infty}\sheafDn(U)$.\end{defn}

\begin{prop} Let $X$ be a manifold without boundary.  There is a natural bijection between $[X,\kzerospec]$ and $\sheafD[X]$\end{prop}
\begin{proof}Postponed to the end of the section.\end{proof}

\begin{defn}Let $\mathcal{F}$ be a sheaf of sets and $X$ a manifold without boundary.  Two sheaf elements $s_0,s_1\;{\in}\;\mathcal{F}(X)$ are concordant if there exists a sheaf element $t\;{\in}\;\mathcal{F}(X{\times}\mathbb{R})$ such that $i_0^*t=s_0$ and $i_1^*t=s_1$ where $i_0$ and $i_1$ are the inclusion maps that identify $X$ with $X{\times}\{0\}$ and $X{\times}\{1\}$ respectively.  The set of concordance classes of elements of $\mathcal{F}(X)$ is denoted $\mathcal{F}[X]$.\end{defn}

So will we prove $[X,\kzerospec]\;{\approxeq}\;\sheafD[X]$.  Contrast this with the following proposition.

\begin{prop}\label{sheafclass}Given a sheaf of sets $\mathcal{F}$ there is a space $|\mathcal{F}|$ which enjoys the property that homotopy classes of maps $[X,|\mathcal{F}|]$ are in natural bijection with $\mathcal{F}[X]$.\end{prop}
\begin{proof} See Appendix A in $\cite{M-W}$ for details\end{proof}.

\begin{defn}The topological realization of $\mathcal{F}$, denoted $|\mathcal{F}|$, is defined to be the realization of the simplicial set
$$[l]\mapsto\mathcal{F}({\Delta}^l_e)$$ where ${\Delta}^{l}_e=\{(t_0,...,t_k)\;{\in}\;\mathbb{R}^{l+1}|{\sum}t_i=1\}$ is the extended l-simplex.\end{defn}

\begin{cor} $\kzerospec$ is weak homotopy equivalent to $|\sheafD|$.\end{cor}
\begin{proof}  Let $X$ range over all spheres and invoke the two previous propositions.\end{proof}

We have seen that $\kzerospec$ is modelled by the realization of an appropriate sheaf.  We might ask if something similar is true for $\Cob$.  Indeed, a result from $\cite{M-W}$ is that there is a sheaf of categories $\sheafC$ which has the property that $B|\sheafC|$ is weak homotopy equivalent to $B\Cob$.   An analogous result in the setting of $\brac$-spaces is proved later in the section.  From here it remains to show that $B|\sheafC|$ is equivalent to $|\sheafD|$.  We appear to be stuck since we are trying to show equivalence between a sheaf of categories and a sheaf of sets.  

Fortunately, $\cite{M-W}$ (appendix A) provides us with two tools to show this equivalence.  The first is a general procedure to take a sheaf of small categories $\mathcal{F}$, and produce an associated sheaf of sets, $\beta\mathcal{F}$ such that there is a weak homotopy equivalence $|\beta\mathcal{F}|{\simeq}B|\mathcal{F}|$.  Here is the recipe for constructing $\beta\mathcal{F}$.  Choose an uncountable indexing set $J$.  An element of $\beta\mathcal{F}(X)$ is a pair $(\mathcal{U},\Psi)$ where $\mathcal{U}=\{U_j|j{\in}J\}$ is a locally finite open cover of $X$, and $\Psi$ is a certain collection of morphisms.  In detail: given a non-empty finite subset $R\subset{J}$, let 
$$U_R\;\;:=\;\;\underset{j{\in}R}{\cap}U_j.$$ 
Then $\Psi$ is a collection ${\varphi}_RS\;{\in}\;N_1\mathcal{F}(U_S)$ indexed by pairs $R\subset{S}$ of non-empty finite subsets of $J$ subject to the conditions\\
i) ${\varphi}_RR=\text{id}_{c_R}$ for an object $c_R\;{\in}\;N_0\mathcal{F}(U_R)$\\
ii)For each non-empty finite $R\subset{S}$, ${\varphi}_{RS}$ is a morphism from $c_S$ to $c_R|U_S$,\\
iii) For all triples ${R}\;\subset\;{S}\;\subset\;{T}$ of finite non-empty subsets of $J$, we have
$${\varphi}_{RT}=({\varphi}_{RS}|U_T){\circ}{\varphi}_{ST}.$$
Theorem 4.1.2 of [MW02] asserts a weak homotopy equivalence\\
$$|\beta\mathcal{F}|{\simeq}B|\mathcal{F}|.$$

The second tool from $\cite{M-W}$ is a useful criteria for when a map between two sheaves of sets induces a weak homotopy equivalence on their realizations.   It is called the $\textit{relative surjectivity criterion}$.  We give a description now.  If $A$ is a closed subset of $X$, and $s\;{\in}\;\text{colim}\mathcal{F}(U)$ where $U$ runs over open neighborhoods of $A$.  Let $\mathcal{F}(X,a;s)$ denote the subset of $\mathcal{F}$ consisting of elements which agree with $s$ in a neighborhood of $A$.  Two elements $t_0,t_1$ are concordant relative to $A$ if they are concordant by a concordance whose germ near $A$ is the constant concordance of $s$.  Let $\mathcal{F}[X,A;s]$ denote the set of such concordance classes.

\begin{prop}$\cite{M-W}\textbf{ Relative Surjectivity Criteria - }$A map $$\tau:\mathcal{F}_1\rightarrow\mathcal{F}_2$$is a weak equivalence provided it induces a surjective map $$\mathcal{F}_1[X,A;s]\rightarrow\mathcal{F}_2[X,A;{\tau}(s)]$$for all $(X,A,s)$.\end{prop}

The rest of the section is devoted to the following tasks:\\
1) Establishing the equivalence $|\sheafD|\;{\simeq}\;\kzerospec$\\
2) Establishing the equivalence $B|\sheafC|\;{\simeq}\;B\Cob$\\
3) Finding a zig-zag of equivalent sheafs, using the $\textbf{Relative Surjectivity Criteria}$.\\

\begin{defn}Let $\sheafDtr$ be a sheaf of categories defined by letting $\sheafDtr(U)$ denote the set of triples $(W,g,a)$ that satisfy:\\
 i) $(W,g)\;{\in}\;\sheafD(U)$,\\
 ii) $a:W\;{\longrightarrow}\;\mathbb{R}$ is smooth,\\
 iii) for each $x\;{\in}\;X$, $f:\pi^{-1}(x)\;{\longrightarrow}\;\mathbb{R}$ is $\brac$-regular (see appendix \ref{appendix1}.\end{defn}

\begin{prop} The forgetful map $\sheafDtr\;{\longrightarrow}\;\sheafD$ is a weak equivalence.\end{prop}

\begin{proof}As stated in the introduction to this chapter, there is a known homotopy equivalence $|{\beta}\sheafDtr|{\simeq}B|\sheafDtr|$ from $\cite{M-W}$, chapter 4.1 so it suffices to show the following lemma.

\begin{lemma}${\beta}\sheafDtr\;{\longrightarrow}\;\sheafD$ satisfies the relative  surjectivity criteria, and thus $|\beta\sheafDtr|\;{\simeq}\;|\sheafD|$.\end{lemma}

\begin{proof} The proof follows exactly as in $\cite{GMTW}$ Proposition 4.2 provided we change ``transverse" to ``$\brac$-transverse".  For the reader's sake, we include the argument here.\\

We must show the forgetful map ${\beta}\sheafDtr\rightarrow{\sheafD}$ satisfies the relative surjectivity condition.  To that end, let $A$ be a closed
subset of $X$, let $(W,g)$ be an element of $\sheafD(X)$, and suppose we are given a lift to ${\beta}\sheafDtr(U')$ of the restriction of $W$ to some open neighborhood $U$ of $A$.  This lift is given by a locally finite open cover $\mathcal{U'}=\{\;U_j\;|\;j{\in}J\;\}$, together with smooth functions $a_R:U_R'{\rightarrow}\mathbb{R}$, one for ecah finite non-empty $R{\subset}J$.  Let $J'{\subset}J$ denote the set of $j$ for which $U_j$ is non-empty, and let $J''=J-J'$.  Now choose a smoothing function $b:X\rightarrow[0,\infty)$ with $A{\subset}\text{Int}{b}^{-1}\{0\}$ and $b^{-1}(0){\subset}U$.  Set $q=\frac{1}{b}:X{\rightarrow}(0,\infty]$.   We can assume that $q(x)>a_R(x)$ for all $R{\subset}J'$ (make $U$ smaller if not).  Now for each $a{\in}X-U$, choose $a{\in}\mathbb{R}$ satisfying:\\
i) $a>q(x)$\\
ii) $a$ is a $\brac$-regular value for $f_x:{\pi}^{-1}(x)\rightarrow\mathbb{R}$.\\
Such an $a$ exists by appendix \ref{appendix1} and furthermore the same value $a$ will satisfy i) and ii) for all $x$ in a small neighborhood $U_x{\subset}X-A$ of $X$, so we can pick an open covering $U''=\{\;U_j\;|\;j{\in}J''\;\}$ of $X-U$, and real numbers $a_j$ such that i) and ii) are satisfied for all $x{\in}U_j$.  The covering $\mathcal{U}''$ may be assumed to be locally finite.  For each finite non-empty $R{\subset}J''$, set $a_R=\text{min}\{a_j|j{\in}R\}$.  For $R{\subset}J(=J'\cup{J}'')$, write $R=R'{\cup}R''$ with $R'{\subset}J'$ and $R''{\subset}J''$, and define $a_R=a_{R'}$ if $R'{\neq}\emptyset$.  This defines smooth functions $a_R:U_R\rightarrow\mathbb{R}$ for all finite non-empty subsets $R{\subset}J$ ($a_R$ is a constant function for $R{\subset}J''$) with the property that $R\subset{S}$ implies $a_S{\leq}a_R|U_S$.  This defines an element of $\beta\sheafDtr(X,A)$ which lifts $W{\in}\sheafD(X)$ and extends the lift given near $A$.\end{proof}

This also ends the Proposition.\end{proof}

For what follows we will need the following convention.  Suppose $X$ is a smooth manifold without boundary and $a_0$,$a_1\:X{\to}\mathbb{R}$
are such that $a_0(x){\leq}a_1(x)$ for all $x{\in}X$. Then $X{\times}(a_0,a_1):=\{\;(x,u)\;{\in}\;X{\times}\mathbb{R}\;|\;a_0(x)<u<a_1(x)\;\}$.
\begin{defn}Given ${\epsilon}>0$ and two smooth functions $a_0,a_1:X\to{\mathbb{R}}$ let $\sheafCtr(X,a_0,a_1,\epsilon)$ be the set of pairs $(W,g)$ where $W$ is a neat $\langle{k}\rangle$-submanifold of $U{\times}(a_0-\epsilon,a_1+\epsilon){\times}\keucNless$ and $g:W{\to}E_N$ is a $\brac$-map which satisfy the criteria below.  In what follows let $\pi$ and $f$ be the projection maps onto $U$ and $\mathbb{R}$ respectively.\\
 i) $\pi$ is a $\brac$-submersion with $d+k$ dimensional fibers,\\
 ii) $(\pi,f)$ is proper,\\
 iii) $(\pi,f):(\pi,f)^{-1}(X{\times}(a_{\nu}-\epsilon,a_{\nu}+\epsilon))\;{\longrightarrow}\;X{\times}(a_{\nu}-\epsilon,a_{\nu}+\epsilon)$ is a $\brac$-submersion for $\nu=0,1$,\\
 iv) $\strucmapN{\circ}g=T^{\pi}W_N$ (see the definition of $\sheafD$ for the definition of $T^{\pi}W_N)$.\end{defn}

\begin{defn}$\textit{Let }C^{tr,\theta}_{d,\langle{k}\rangle}(X,a_0,a_1):=\underset{{\epsilon}\rightarrow{0}}{\text{colim}}C_{d,\brac}^{tr,\theta}(X,a_0,a_1,\epsilon)$.\end{defn}

\begin{defn}$\text{mor}\sheafCtr(X):={\coprod}\sheafCtr(X,a_0,a_1)$ ranging over all pairs $(a_0,a_1)$ of smooth functions where $a_0\;{\leq}\;a_1$ and the set of $x\;{\in}\;X$ suc that $a_0(x)=a_1(x)$ is a (possibly empty) union of connected components of $X$.\end{defn}

Two morphisms $s\;{\in}\;\sheafCtr(X,a_0,a_1)$ and $t\;{\in}\;\sheafCtr(X,a_1,a_2)$ are composable if there exists some
$\epsilon:X{\to}(0,\infty)$ and $u\;{\in}\;\sheafCtr(X,a_1-\epsilon,a_1+\epsilon)$ such that
$s=u\;{\in}\;\sheafCtr(a_1-\epsilon,a_1)$ and $t=u\;{\in}\;\sheafCtr(a_1,a_1+\epsilon)$.  In this case composition
of the morphisms is given by taking the union of the two submanifolds and gluing the appropriate maps.  The objects of
$\sheafCtr(X)$ are identified with the set of identity morphisms, $\sheafCtr(X,a_0,a_0)$.\\

\begin{defn}Let $\sheafC(X)$ be the subcategory of $\sheafCtr(X)$ consisting of morphisms  $(W,g,a,b)\;{\in}\;\sheafCtr(X,a,b)$ that satisfy two additional properties:\\
i) ${\pi}^{-1}(x)$ is a neat $<$k+1$>$-submanifold of $\{x\}{\times}[a,b]{\times}\mathbb{R}^k_+{\times}\mathbb{R}^{d-1+\infty}$,\\
ii) In a neighboorhood of ${\partial}_{k+1}{\pi}^{-1}(x)$ which we write as ${\partial}_{k+1}{\pi}^{-1}(x){\times}[0,\epsilon)$, $g={\partial}_{k+1}g{\times}\text{id}_{[0\epsilon)}$  (${\partial}_{k+1}$ is in the $[a,b]$ direction). 
\end{defn}

\begin{prop}The inclusion functor $\sheafC(X)\;{\longrightarrow}\;\sheafCtr(X)$ is a weak equivalence.\end{prop}

\begin{proof}The proof in $\cite{GMTW}$ (proposition 4.4 page 18) adapts easily to the $\brac$-space
setting.\end{proof}

\begin{prop}There is an equivalence $\sheafDtr\;\longrightarrow\;\sheafCtr$.\end{prop}

\begin{proof}Again, the proof in $\cite{GMTW}$ (Proposition 4.3 page 18) adapts immediately to the $\brac$-space setting.\end{proof}

\begin{prop}$B\Cob\;\;{\simeq}\;\;B|\sheafC|$\end{prop}

\begin{proof}It suffices to show that for each $p\;{\geq}\;0$, the spaces $N_p\Cob$ and $N_p|\sheafC|$ are weak homotopy equivalent.  By a theorem of Milnor $\cite{Milnor}$, any space is weakly homotopic to the geometric realization of its singular set, so
$$N_p\Cob\;\;{\simeq}\;\;|\;[l]{\mapsto}C({\Delta}^l,N_p{\Cob})\;|.$$
Now we argue that the singular set on the right is equivalent to a smooth singular set; we outline what we mean by smooth.  Recall from definition \ref{defncobobs} that $Y(M)$ is a subset of $$\text{Emb}_{\brac}(M,\mathbb{R}^k_+\times\mathbb{R}^{d+n}){\times}C_{\brac}(M,\struc)$$
so there is a projection map $\text{proj}:\frac{Y(M)}{\text{Diff}_{\brac(M)}}\;\longrightarrow\;\frac{\text{Emb}_{\brac}(M,\mathbb{R}^k_+\times\mathbb{R}^{\infty})}{\text{Diff}_{\brac(M)}}$.   For a smooth manifold $X$, we declare $u:X\;{\longrightarrow}\;\text{ob}\Cob$ to be smooth if the 
projections onto $\frac{\text{Emb}_{\brac}(M,\mathbb{R}^k_+\times\mathbb{R}^{\infty})}{\text{Diff}_{\brac(M)}}$ and $\mathbb{R}$ are smooth.  By smoothing theory, we see that $C^{\infty}({\Delta}^l_e,N_{p}\Cob)$ is weakly homotopic to $C({\Delta}^l_e,N_{p}\Cob)$ in the case $p$ equals 0.  The argument for higher $p$ is identical.  Thus
 
$$|\;[l]{\mapsto}C({\Delta}^l,N_p{\Cob})\;|\;\;{\simeq}\;\;|\;[l]{\mapsto}C^{\infty}({\Delta}^l_e,N_p{\Cob})\;|$$

Recall that the object space for $\Cob$ is $$\underset{[M]\in\mathcal{S}}{\coprod}\;\frac{Y(M)}{\text{Diff}_{\brac}(M)}{\times}\mathbb{R}.$$
Now suppose  $X$ is a manifold without boundary and $$u=u_1{\times}u_2:X\rightarrow\underset{[M]\in\mathcal{S}}{\coprod}\;\frac{Y(M)}{\text{Diff}_{\brac}(M)}{\times}\mathbb{R}.$$ is a smooth map.  Then $(\;\text{graph}(u_1)\;,\;g{\circ}\pi{_2}\;)$ is an element of $\text{ob}\sheafC(X)$ (where $\pi{_{2}}$ is projection on the graph of $u_1$).
Conversely, given an object $(W,g)\;{\in}\;\text{ob}\sheafC(X)$ we obtain a smooth map $X\;{\longrightarrow}\;\text{ob}\Cob$ by sending $x\;{\in}\;X$
to $(\;{\pi}^{-1}(x)\;,\;g{_{|{\pi}^{-1}(x)}}\;)$.  Thus we've shown that
$$C^{\infty}(X,N_{p}\Cob)\;{\approxeq}\;N_p\sheafC(X)$$
when $p$ equals $0$.  The cases $p>0$ are similar.  Setting $X={\Delta}^l_e$ yields
$$|\;[l]{\mapsto}C^{\infty}({\Delta}^l_e,N_p{\Cob})\;|\;\;{\simeq}\;\;|N_p\sheafC|$$
As each of the three equivalences is simplicial with respect to the $p$ variable, we may string them together to obtain 
$$B\Cob\;\;{\simeq}\;\;|\;[p]\mapsto|N_p\sheafC|\;|=B|\sheafC|.$$
\end{proof}

For the reader's convenience here is a summary of the zig-zag of weak homotopy equivalences that combine to yield the theorem.

$$B\Cob\overset{(1)}{\rightarrow}{B|\sheafC|}\overset{(2)}{\rightarrow}B|{\sheafCtr}|\overset{(3)}{\leftarrow}{B|\sheafDtr|}\overset{(4)}{\rightarrow}|{\beta}\sheafDtr|\overset{(5)}{\rightarrow}{|\sheafD|}\overset{(6)}{\rightarrow}\kzerospec.$$
Informally, the reader should think of elements of $\sheafD$ as being parametrized $d+k$ dimensional $\brac$-manifolds, $W$, along with a proper map to $\mathbb{R}$ (the ``cobordism direction").
Going from right-to-left we convert these elements until they are parametrized morphisms of $\Cob$.  To that end, in (5) we added in the extra information of a parametrized beginning and parametrized end slice which $W$ meets transversally.  Then we discarded the rest of the manifold outside the slices in (3), and
finally with (2) we insisted that $W$ meets the boundary slices not just transversely, but perpendicular to the $\mathbb{R}$ ``cobordism" direction.  To prove maps (2)-(5) are weak homotopy equivalences we made significant use of the $\textbf{Relative Surjectivity Criterion}$.  That left the maps (1) and (6).  The first map was loosely a Yoneda embedding.  The last map is based on the Pontrjagin-Thom construction.  That it is an equivalence is to be shown.\\ 

\begin{prop}$|\sheafD|\;\;{\simeq}\;\;\kzerospec$
\end{prop}

\begin{proof}From proposition \ref{sheafclass} it suffices to construct a bijection $\rho$ and its inverse $\sigma$ between relative concordance classes $\sheafD[X,A;s_0]$ and $[X,A;\kzerospec,s_0']$
for any closed manifold $X$.  Thus, by allowing our choice of $X$ to range over spheres of arbitrary dimension, and letting $A$ be a point we obtain our result.  The relative version follows along lines similar to the absolute case, which we now present. \\

We construct $\sigma$ first.  Let $X$ be a closed manifold and let us pick a map $h$ represeneting the class $[h]\;{\in}\;[X,\kzerospec]$.  As $X$ is compact, it lifts to a map

$$h:X\;\longrightarrow\;\Omega^{d-1+n}_{\brac}{\thomN}$$
for some $N>0$.  Its adjoint map
$$h_{ad}:X_+{\wedge}(\keucNless)^c{\rightarrow}\thomN.$$\\
is a $\brac$-map.  After removing the point at infinity from $X_+{\wedge}(\keucNless)$ we are left a $\brac$-space which we identify with $$X{\times}\keucNless.$$  
Because the inverse image under $h$ of the total space ${\theta}_N^*{\gamma}^{\perp}(d,n)\brac=\thomN-*$ is an open subspace
of $X{\times}\keucNless$ and $X{\times}\keucNless$ is a $\brac$-manifold, $h^{-1}({\theta}_N{\gamma}^{\perp}(d,n))$ is also a $\brac$-manifold.  Also, $h$ restricts to a map

$$h:\;\;h^{-1}({\theta}_N^*{\gamma}^{\perp}(d,n))\;\longrightarrow\;{\theta}_N^*{\gamma}^{\perp}(d,n).$$

We are now in a situation where we may apply appendix \ref{appendix2}.  Thus, ${\theta}_Nh$ may be perturbed under a homotopy $l$ to a smooth map ${h}_0$ of $\brac$-manifolds transverse to $G(d,n)\brac$.  Set
$$V:={h}_0^{-1}(M)$$
which is a codimension n $\brac$-submanifold of $X{\times}\keucNless$ which up to concordance is a neat $\brac$-submanifold.  Let ${\pi}$ be the projection from $V$ onto $X$ and $\iota$ the inclusion map of $V$ into $X{\times{\keucNless}}$.

By construction, the normal bundle $\nu$ of $V$ in $X{\times{\keucNless}}$
is ${h}_0^*{\gamma}^{\perp}(d,n)\brac$.  Set
$$T^{\pi}X:={h}_0^*{\gamma}^{\perp}(d,n)\brac.$$
Then we have the following $\brac$-bundle isomorphism:
$$\nu{\oplus}T^{\pi}X\;\;=\;\;{h}_0^*\geoperpcanon{\oplus}{h}_0^*\geocanon\;\;{\simeq}\;\;{h}_0^*{\epsilon}^{d+n}\brac.$$
On the otherhand from the fact that $V$ is a submanifold of $X{\times{\keucNless}}$ we also have the $\brac$-bundle isomorphism:
$$TV{\oplus}\nu\;\simeq\;i^*T(X{\times}\keucNless)\;\simeq\;i^*TX{\oplus}i^*T(\keucNless)\;=\;i^*TX{\oplus}{\epsilon}^{d-1+n}\brac.$$
Piecing these two lines together gives an isomorphism
$$TV\brac{\oplus}{\epsilon}^{d+n}\brac\;\simeq\;TV{\oplus}\nu{\oplus}T^{\pi}X\;{\simeq}\;\;i^*TX{\oplus}{\epsilon}^{d-1+n}\brac{\oplus}T^{\pi}X\brac.$$
By appendix \ref{appendix5}, this isomorphism is induced up to $\brac$-homotopy by a unique $\brac$ isomorphism $${\Phi}:TV\brac{\oplus}{\epsilon}\;\overset{\simeq}{\longrightarrow}\;{\pi}^*TX{\oplus}T^{\pi}X\brac.$$\\
Now set 
$$W:=V{\times}\mathbb{R},$$
let $p:W\to{V}$ be projection, and consider
$${\pi}^*\circ{\Phi}{\circ}p^*:TW=p^*TV{\oplus}\epsilon\;\longrightarrow\;p^*{\pi}^*TX{\oplus}p^*T^{\pi}X\;\longrightarrow\;TX$$
which is evidently a bundle surjection and induced by projection onto $X$.  By the Philips Submersion Theorem (appendix \ref{appendix4}), ${\pi}^*\circ{\Phi}{\circ}p^*$ is homotopic to a $\brac$-submersion $s$ through some homotopy $s_t$.  Unfortunately, $s$ is now no longer the same map as the projection of $W$ onto $X$.   To remedy this, pick an immersion $e'$ be of $W$ into $\mathbb{R}^{n'}$, and $\phi:I\to{I}$ a smooth monotone increasing function that is zero in a neighborhood of zero and one in a neighborhood of one.  Then
$$\text{id}_W{\times}\phi(t)e':W{\times}I\;\longrightarrow\;X{\times}\mathbb{R}{\times}\mathbb{R}^{d-1+n+n'}$$
is a $\brac$-isotopy from the identity embedding of $W$ to an embedding $id_W{\times}e_1$ where $e_1$ is also a $\brac$-embedding.  Follow this isotopy by $s_t{\times}e_1$.  Then 
$$W_0:=s_1{\times}e_1(W)$$
 is a $\brac$-submanifold of $X{\times}\mathbb{R}{\times}\mathbb{R}^{d-1+n+n'}$ such that projection onto $X$ is a submersion.  The reader should notice that we may assume the projection $f$ of $W_0$ onto the first $\mathbb{R}$ factor is proper.\\

To produce the map of tangential data $g:W_0{\to}{\theta}_N^*{\gamma}^{\perp}(d,n)\brac$.  Since $\theta_N$ is a fibration, it suffices to find a map $\overline{g}:W_0{\to}{\theta}_N^*{\gamma}^{\perp}(d,n)\brac$ and a homotopy from ${\theta}_N\overline{g}$ to $T^{\pi}W_0$.  The map $\overline{g}$ is provided by the composition 
$$W_0\;\overset{s_1{\times}e_1}{\longrightarrow}\;W\;\overset{p}{\longrightarrow}\;V\;\overset{h_{|V}}{\longrightarrow}\;{\theta}_N^*{\gamma}^{\perp}(d,n)\brac.$$
Recall that $h$ is the map from the very beginning of our construction, and we have chosen a homotopy $l$ from ${\theta}_Nh$ to $h_0$.  Also note that $h_0{\circ}p$ is precisely the map classifying the normal bundle of $W$ in $X{\times}\mathbb{R}{\times}\keucNless$ and homotopy $s$ and the $\brac$-isotopy from the previous paragraph provide a homotopy from $h_0{\circ}p$ to $T^{\pi}W_0$ to $T^{\pi}W_0$.\\

  Now we construct $\rho$.  Suppose we are given $(W,g)\;{\in}\;\sheafD[X]$.  Recall that $W$ is a $\brac$-submanifold of $X{\times}\mathbb{R}{\times}\keucNless$.  By appendix \ref{appendix1} there is a regular value $c\;{\in}\;\mathbb{R}$ for the projection map $f:W\;\rightarrow\;\mathbb{R}$.  Set 
$$M:=f^{-1}(c).$$
$M$ is a $\brac$-submanifold of $X{\times}\{c\}{\times}\keucNless$.  The normal bundle $\mu$ of $M$ in $X{\times}\{a\}{\times}\keucNless$ is a $\brac$-vector bundle and for any $\twokob$ the Thom collapse map produces a map 
$$X_+\;{\wedge}\;(\mathbb{R}^k(a){\times}\mathbb{R}^{d-1+n})^c\;\longrightarrow\;M^{\mu}(a).$$
These maps assemble to form a $\brac$-map
$$X_+\;{\wedge}\;(\mathbb{R}^k_+{\times}\mathbb{R}^{d-1+n})^c\;\longrightarrow\;M^{\mu}\brac.$$

After picking a metric on $X$, the normal bundle $\omega$ of $W{\subset}X{\times}\mathbb{R}{\times}\keucNless$  fits in the following short exact sequence of vector bundles
$$
0\;\longrightarrow\;TW\overset{\iota}{\longrightarrow}\;T(X{\times}\mathbb{R}{\times}\keucNless)\;\longrightarrow\;\omega\;\longrightarrow\;0
$$
which is canonically isomorphic to 
$$
0\;\longrightarrow\;TW\;\overset{\iota}{\longrightarrow}\;{\pi}_1^*TX\;{\oplus}\;{\pi}_2^*T(\mathbb{R}{\times}\keucNless)\;\longrightarrow\;\omega\;\longrightarrow\;0
$$
Since ${\pi}^*:TW\;\rightarrow\;TX$ is a bundle surjection, and $\pi$ is nothing more than ${\pi}_1{\circ}\iota$ the sequence is isomorphic to
$$
0\;\longrightarrow\;\text{ker}{\pi}^*\;\overset{{\pi}_2{\iota}^*}{\longrightarrow}\;T(\mathbb{R}{\times}\keucNless)\;\longrightarrow\;\omega\;\longrightarrow\;0$$
This identifies the normal bundle $\omega$ as the total space of the bundle ${{\pi}_2\iota}^*{\gamma}^{\perp}(d,n)\brac$.  The product of this map with $g$ yields a $\brac$-map 
$$G:\omega\;\longrightarrow\;{\theta}_N^*\gamma^{\perp}(d,n)\brac$$
An argument identical to the one above shows that the normal bundle $\omega$ restricted to $M$ is isomorphic to  $\mu$ (replace $X$ and $\pi$ with $\mathbb{R}$ and $f$).  This yields the following sequence of maps
$$
\mu\;\overset{\simeq}{\longrightarrow}\;{\omega}_{|M}\;\overset{i}{\longrightarrow}\;{\omega}\;\overset{G}{\longrightarrow}\;{\theta}_N^*\gamma^{\perp}(d,n)\brac
$$
The composition of this map induces a map on the Thom space
$$M^{\mu}\;\longrightarrow\;\thomN\brac.$$
Precomposing this map with the Thom collapse map and taking the adjoint with respect to $X$ produces a map
$$X\;\rightarrow\;\Omega^{d-1+n}_{\brac}B^{{\theta}_N^*{\gamma}^{\perp}(d,n)}.$$

   A check shows that $\rho$ and $\sigma$ are homotopy inverse to each other. \end{proof}

\section{Applications}

$\textbf{Unoriented and Oriented }<\textbf{k}>\textbf{-manifolds}$
\begin{prop}\label{applic1}$B\text{Cob}_{d,\brac}\;\;{\simeq}\;\;{\Omega}^{{\infty}-1}\underline{X}\textit{ where }\underline{X}\text{ is }$\\
$\textit{the following suspension spectrum: }$\\
$$BO(d)^{-{\gamma}_d}\textit{ when }k=0$$
$${\Sigma}^{\infty}BO(d+1)_+\textit{ when }k=1$$
$${\Sigma}^{\infty}(\underset{j=1}{\overset{k}{\bigvee}}\overset{{k\text{-}1}\choose{j\text{-}1}}{\bigvee}BO(d+j)_+)\textit{ when }k\;{\geq}\;2$$\end{prop}
Let $\theta:G(d,n)^+\brac{\rightarrow}G(d,n)\brac$ be the structure map from the oriented $\brac$-Grassmanian to the $\brac$-Grassmanian which forgets orientation.
\begin{prop}\label{applic1'}$B\text{Cob}_{d,\brac}^+\;\;\simeq\;\;{\Omega}^{{\infty}-1}\underline{X}\textit{ where }\underline{X}\text{ is }$\\
$$BSO^{-\theta^*{\gamma}_d}\textit{ when }k=0$$
$${\Sigma}^{\infty}BSO(d+1)_+\textit{ when }k=1$$
$${\Sigma}^{\infty}(\underset{j=1}{\overset{k}{\bigvee}}\overset{{k\text{-}1}\choose{j\text{-}1}}{\bigvee}BSO(d+j)_+)\textit{ when }k\;{\geq}\;2$$\end{prop}

\begin{proof}
 
When $k=0$, there is nothing to prove, since the homotopy colimit is over the trivial category with one object.
When $k=1$, the cofiber of\\
 $${\Sigma}^{-1}BO(d)^{-{\gamma}_d}{\rightarrow}BO(d\text{+}1)^{-{\gamma}_{d+1}}$$ is known to be ${\Sigma}^{\infty}BO(d+1)_+$, and the cofiber of
 $${\Sigma}^{-1}BSO(d)^{-\theta^*{\gamma}_d}{\rightarrow}BSO(d\text{+}1)^{-\theta^*{\gamma}_{d+1}}$$ is known to be ${\Sigma}^{\infty}BSO(d+1)_+$.\\

It remains to evaluate the homotopy colimit when $k\;{\geq}\;2$ where the homotopy colimit is either $\underset{\underline{2}^k_*}{\text{hocolim}}\;\plainthom\brac_*$ in the case of propostion \ref{applic} or $\underset{\underline{2}^k_*}{\text{hocolim}}\;{G^+}^{-{\theta}^*{\gamma}_{d}}\brac_*$ for proposition \ref{applic'}.  The arguments in both cases are identical; we demonstrate the computation in the former case.  The reader may briefly want to glance at $\textbf{section 2}$ for the definition of ${\partial}_k$ and $\overline{{\partial}_k}$.\\

For any $\brac$-space $E\brac$, the homotopy colimit $\underset{\underline{2}^k_*}{\text{hocolim}}\;E_*$ is homotopic to the homotopy cofiber of the map
$${\smallhoco}\overset{f}{\longrightarrow}\bighoco$$
where $f$ is induced by the composition of maps
$$\xymatrix{
E(a)\ar[r]&E(a\text{+}e_k)\ar[r]&\bighoco
}
$$
\begin{lemma}When $E\;=\;\plainthom\brac$, $f$ is nullhomotopic.\end{lemma}
\begin{proof}
Note that for any $\brac$-space $E$, $\text{hocolim}_{\underline{2}^k_*}E_*-\text{im}E(1)$ deformation retracts onto $\text{hocolim}_{\underline{2}^k_*-1}E$ which is contractible ($\underline{2}^k_*-1$ is the full subcategory of $\underline{2}^k_*$ after we omit the object $1$.  It has an initial object).  Thus it suffices to show that when $E=\plainthom\brac$, $f$ is homotopic to a map which factors through $\bighocothom-\text{im}\plainthom(1)$.  We construct this homotopy as follows.
Let
$$R_{\theta}:=\begin{bmatrix}\text{cos}\frac{\pi}{2}\theta &\ -\text{sin}\frac{\pi}{2}\theta\\\text{sin}\frac{\pi}{2}\theta &\ \text{cos}\frac{\pi}{2}\theta\end{bmatrix}$$
and let

$$M_{\theta}:=\left[\array{c|c}I_{k-2} &\ 0\\\hline{0} &\ R_{\theta}
\endarray\right]$$
for $\theta\;{\in}\;[0,1]$ written with respect to the standard basis $(e_1,...,e_k)$.
Suppose $a\;{\in}\;{\partial}_k\underline{2}^{k-1}$ (i.e. is of the form $(...,0)$).  For each $\theta$, $M_{\theta}$ induces a map $M_{\theta}^*:\plainthom(a)\;{\rightarrow}\;\plainthom(a$+$e_k)$ and the following diagram commutes for all $a<b\;{\in}\;{\partial}_k\underline{2}^{k-1}$:
$$\xymatrix{
\plainthom(a)\ar[d]\ar[r]^{M_{\theta}^*}&\plainthom(a\text{+}e_k)\ar[d]\\
\plainthom(b)\ar[r]^{M_{\theta}^*}&\plainthom(b\text{+}e_k)
}$$
There are of course also the canonical homotopies
from $\plainthom(a+e_k)\rightarrow\bighocothom$ to the composition 
$$\plainthom(a\text{+}e_k)\;\;{\longrightarrow}\;\;\plainthom(b\text{+}e_k)\;\;\longrightarrow\;\;{\bighocothom}.$$
Gluing these triangles onto the previous squares yields a diagram commuting up to a prescribed homotopy:

$$\xymatrix{
\plainthom(a)\ar[d]\ar[dr]&\\
\plainthom(b)\ar[r]&\bighocothom
}
$$
for all $a<b\;{\in}\;{\partial}_k\underline{2}^{k-1}$.  From this we see that
$M_{\theta}$ actually induces a map
$$F_{\theta}:I{\times}\smallhocothom\;\;\longrightarrow\;\;\bighocothom.$$
When $\theta=0$, $M_0^*=\plainthom(a<a+e_k)$ and so $F_{0}=f$.  For $\theta=1$, after identifying $\plainthom(a)$ with $\plainthom(a')$ where $a'=a-e_{k-1}+e_k$, we see that $M_1^*=\plainthom(a'<a+e_k)$ and so $F_{1}$ factors through $\underset{{\underline{2}^k_*}}{\text{hocolim}}\plainthom-\text{im}\plainthom(1)$ as desired.\end{proof}

Applying the lemma gives
$$\underset{\underline{2}^k_*}{\text{hocolim}}{\plainthom}\;\simeq\;{\Sigma}\;\smallhocothom\;\vee\;\bighocothom$$
which after unraveling the definition of $\geoperpcanon$ we see is precisely
$${\simeq}\;{\Sigma}\;\underset{\underline{2}^{k-1}_*}{\text{hocolim}}\;{\Sigma}^{-1}G^{-{\gamma}_d}<\text{k-1}>_*\;\vee\;\underset{\underline{2}^{k-1}_*}{\text{hocolim}}\;G^{-{\gamma}_{d+1}}<\text{k-1}>_*.$$

By induction on $k$ we have already identified the homotopy type of the righthand side.  The result follows from repeated application of Pascal's Triangle identity.\end{proof}

$\textbf{2 - The restriction functors}$
Let $0\;{\leq}\;l<k$, $0\;{\leq}\;d$, and $c\;{\in}\;\underline{2}^k$ with $|c|=l$.  The functor $\underline{2}^l\overset{\simeq}{\rightarrow}\underline{2}^k(c)\hookrightarrow\underline{2}^k$ induces a restriction functor
$${\Psi}:\text{Cob}_{d,\brac}^{\theta}\;\longrightarrow\;{\text{Cob}}_{d,\langle{l}\rangle}^{\theta}$$
which sends $(W,g,a,b)$ to $(W(c),g(c),a,b)$.  We focus on the case when $k=1$ and $l=0$.  From the previous section we know that $\Psi$ induces a map on classifying spaces
$$B{\Psi}:{\Omega}^{\infty-1}{\Sigma}^{\infty}BO(d+1)_+\;\longrightarrow\;{\Omega}^{\infty-1}MT(d)$$

\begin{thm}
Let $k=1(l=0)$ and $\theta$ be the trivial or orientation fibration.  Then up to homotopy ${\Psi}$
 is induced by the spectrum map ${\Sigma}^{\infty}BO(d+1)\;\longrightarrow\;MT(d)$ which occurs in the cofibration sequence $\Sigma^{-1}MT(d)\;\rightarrow\;MT(d+1)\;\rightarrow\;{\Sigma}^{\infty}BO(d+1)$ described in $\cite{GMTW}$.
When $d=k=1$, this has been identified as the circle transfer.
\end{thm}
\begin{thm}
If $l>0$, and $\theta$ is the identity or orientation fibration, notice from proposition \ref{applic1} (and analogously for proposition \ref{applic1'}) that
$B\Psi$ is a map from $$\Omega^{\infty}(\text{spectrum underlying }B\text{Cob}_{d,\langle{k-1}\rangle}\;{\bigvee}\;\text{spectrum underlying }B\text{Cob}_{d+1,\langle{k-1}\rangle})$$ $$\text{to  }\Omega^{\infty}(\text{spectrum underlying }B\text{Cob}_{d,\langle{k-1}\rangle}).$$  $B\Psi$ is induced by the spectrum map which is the identity on the first summand and collapses the second summand to a point. 
\end{thm}

\begin{proof}Notice the following diagram is a commutative diagram of topological categories:
$$\xymatrix{
\Cob\ar[d]_{\Psi}\ar[r]&\mathcal{P}(\Omega^{\infty-1}_{\langle{k}\rangle}B^{-\theta{\gamma}_d}\langle{k}\rangle)\ar[d]\\
\text{Cob}_{d,\langle{l}\rangle}^{\theta}\ar[r]&\mathcal{P}(\Omega^{\infty-1}_{\langle{l}\rangle}B^{-\theta{\gamma}_d}\langle{l}\rangle).
}$$
Here $\mathcal{P}$ is the functor which assigns to any topological space its path-space category.  The horizontal maps are defined by sending a morphism $(W,g,a,b)$ to the path which at time $a\;{\leq}\;t\;{\leq}\;b$ is the composition of the Thom collapse of $(\mathbb{R}^{\infty}{\times}\mathbb{R}^k_+{\times}\{t\})^c$ to the tubular neighborhood of the subspace of $W$ lying over $t$ with the map $g^*$.  On classifying spaces, this map induces the homotopy equivalence of \ref{mainthm}.  The right vertical map is simply restriction.\\

Now we focus on the special case where $\theta$ is the identity, or the orientation map, $k=1$ and $l=0$.  Then we have another diagram which commutes up to homotopy:
$$\xymatrix{
\Omega^{\infty-1}_{\langle{1}\rangle}BO(d)^{-\theta{\gamma}_d}\langle{1}\rangle\ar[r]\ar[d]&\Omega^{\infty-1}\text{hocolim}({\Sigma}^{-1}BO(d)^{-\theta{\gamma}_d}\rightarrow{BO(d+1)^{-\theta{\gamma}_{d+1}}})\ar[d]\\
\Omega^{{\infty}-1}BO(d)^{-\theta{\gamma}_d}\ar[r]^{=}&\Omega^{{\infty}-1}BO(d)^{-\theta{\gamma}_d}\\
}.$$
The vertical map on the left is in the same homotopy class as the map on classifying spaces induced by the right vertical map in the previous square.
The horizontal top map is the equivalence from proposition \ref{PuppeLaures}, which in this case reduces to the fact that the homotopy fiber of a map between zero spaces of spectra is homotopic to the zerosection of the homotopy cofiber of the corresponding map.  It follows that defining the right vertical map to be the map in the Puppe sequence $$\text{hocolim}({\Sigma}^{-1}BO(d)^{-{\gamma}_d}\rightarrow{BO(d+1)^{-{\gamma}_{d+1}}})\;\longrightarrow\;{\Sigma}{\Sigma}^{-1}BO(d)^{-{\gamma}_d}=BO(d)^{-{\gamma}_d}$$
yields a homotopy commutative diagram.\\

It is known ($\cite{CircTran}$) that when $d=k=1$, $l=0$, and $\theta$ is trivial or an orientation, the map on the right is induced by a circle transfer.  Thus in the oriented case, we have identified this map as being up to homotopy the map induced by the circle transfer
$$Q({\Sigma}\mathbb{CP}^{\infty}_+)\;\longrightarrow\;QS^0$$

For general $0\;{\leq}\;l<k$, notice that any restriction functor is the composition of restriction functors where $l=k-1$; thus it suffices to consider this special case.  We have just dealt with the case $k=1$.  Now we investigate larger $k$ and find the maps to be less interesting.  The argument precedes as before, so we find a homotopy commutative diagram:
$$\xymatrix{
B\Cob\ar[d]_{\Psi}\ar[r]&\Omega^{\infty-1}\Sigma^{\infty}\text{hocofib}
({\Sigma}^{-1}(\underset{j=1}{\overset{l}{\bigvee}}\overset{{l-1}\choose{j\text{-}1}}{\bigvee}BO(d+j)_+)\rightarrow\underset{j=2}{\overset{k}{\bigvee}}\overset{{k\text{-}1}\choose{j\text{-}1}}{\bigvee}BO(d+j)_+)\ar[d]\\
B\text{Cob}_{d,\langle{k-1}\rangle}\ar[r]&\underset{j=1}{\overset{l}{\bigvee}}\overset{{l-1}\choose{j\text{-}1}}{\bigvee}BO(d+j)_+
}.$$
We've already established the cofiber is of a nullhomotopic map, so the vertical  map on the right is given by $\text{id}\;{\vee}\;*$.\end{proof}

$\textbf{3 - Maps To A Background Space}$\\

Let $X$ be any topological space and let $\fullstruc$ be tangential data.  Then consider the category $\Cob(X)$ whose objects are pairs $(\alpha,x)$ where $\alpha=(N,a,f)$ is an object in $\Cob$ and $x$ is a map from $N$ to $X$.  Likewise its morphisms are pairs $(\beta,x)$ where $\beta=(W,a,b,F)$ is a morphism in $\Cob$ and $x$ is a map from $W$ to $X$. 

\begin{thm}$$B\Cob(X)\;\;{\simeq}\;\;X_+\;{\wedge}\;\underset{\underline{2}^{k}_*}{\text{hocolim}}{\thom}_*\;({\simeq}X_+{\wedge}B\Cob).$$\end{thm}
\begin{proof}

 The category $\Cob(X)$ is clearly isomorphic to the category $\text{Cob}_{d,\brac}^{{\strucmap}p}$ where $p$ is the projection of $X{\times}\struc$ onto $\struc$ composed with $\strucmap$.  Then the $\textbf{Main Corollary 4.5}$ yields
$$B\text{Cob}_{d,\brac}^{{\strucmap}p}\;\;\simeq\\;\;{\Omega}^{{\infty}-1}\;\underset{\underline{2}^k_*}{\text{hocolim}}\;(X_+{\wedge}{\thom})_*.$$
A quick check using the universal property of the homotopy colimit shows that $$\underset{\underline{2}^k_*}{\text{hocolim}}\;(X_+\;{\wedge}\;{\thom})_*\;\;{\simeq}\;\;X_+\;{\wedge}\;\underset{\underline{2}^k_*}{\text{hocolim}}\;{\thom}_*.$$
Hence we conclude
$$B\Cob(X)\;\;=\;\;B\text{Cob}_{d,\brac}^{{\strucmap}p}\;\;{\simeq}\;\;X_+\;{\wedge}\;\underset{\underline{2}^k_*}{\text{hocolim}}\;{\thom}_*.$$
\end{proof}

\section{Appendix}

We extend some definitions and facts (mostly) from differential topology 
to the $\brac$-space setting.\\

\begin{append}\label{appendix1}$\textbf{ - Regular values for }\langle\textbf{k}\rangle\textbf{-manifolds}$\end{append}

\begin{defn}Let $M$ be a $d$-dimensional $\brac$-submanifold of $\mathbb{R}^k_+{\times}\mathbb{R}^N$ for some $N,k\;{\geq}\;0$, let $X$ be a plain old closed $l$-dimensional manifold promoted to a $\brac$-space, and let $f:M\;{\rightarrow}\;X$ be a smooth $\brac$-map.  We say $c\;{\in}\;X$ is a regular value if $c$ is a regular value of $f(a)$ for each $\twokob$.\end{defn}

$\textbf{Fact i)}$ If $c\;{\in}\;N$ is a $\brac$-regular value, then $f^{-1}(c)$ is a $\brac$-submanifold of $M$ of codimension one.\\
\newline
$\textbf{Fact ii)}$ If $f(1)$ is closed (in particular, if it is proper then it is closed), then the set $Z$ of $\brac$-regular values is an open subset of $N$ and $N-Z$ has measure 0.\\

\begin{proof} For $\twokob$, let $Z(a)$ be the regular values of $f(a):M(a)\;\rightarrow\;{N}$.  Then by Sard's theorem, $Z(a)$ is open and $N-Z(a)$ has measure zero.  As $Z=\cap_{\twokob}Z(a)$, the statement follows.\end{proof}

\begin{append}\label{appendix2}$\textbf{ - Transversality for }\langle\textbf{k}\rangle\textbf{-manifolds}$\end{append}
\begin{defn}A map of $\brac$-manifolds $f:M\;{\rightarrow}\;N$ is $\brac$-transverse to a $\brac$-submanifold $S$ of $N$ if for each $\twokob$, $f(a):M(a)\;{\rightarrow}\;N(a)$ is transverse to $S(a)$.\end{defn}

 $\textbf{Observation: }$
 Let $f:M\;{\rightarrow}\;N$ be a $\brac$-map that is transverse to a neat $\brac$-submanifold $S$ of $N$.  Then $f^{-1}(S)$ is a $\brac$-submanifold of $M$ of codimension equal to the codimension of $S$.

Let $M$ and $N$ be compact $\brac$-manifolds and let $S$ be a neat $\brac$-submanifold of $N$.  Let $C^{\infty}(M,N)$ denote the space of smooth
 $\brac$-maps from $M$ to $N$.  Let $T_SC^{\infty}{\brac}(M,N)$ denote the subspace of $C^{\infty}_{\brac}(M,N)$ consisting of those maps $\brac$-transverse to $S$.  

\begin{lemma}\label{translemma}$\textbf{ - Thom Transversality for }\langle\textbf{k}\rangle\textbf{-manifolds:}$
$$T_SC^{\infty}_{\brac}(M,N)\textit{ is dense (and open) in }C^{\infty}_{\brac}(M,N).$$\end{lemma}

\begin{proof} Fix $a{\in}\underline{2}^k$ and let ${\sigma}_a$ be the composition of the maps
$$
C^{\infty}_{\brac}(M,N)\;\longrightarrow\underset{\{(a,b){\in}\underline{2}^k{\times}\underline{2}^k|a{\leq}b\}}{\prod}C^{\infty}(M(a),N(b)){\longrightarrow}\;C^{\infty}(M(a),N(a))$$
where the second map is just projection and the first map is the canonical inclusion map (the domain is defined and topologized as a limit over the product space in the middle).  The map ${\sigma}_a$ is continuous since both the inclusion map and projection map are continuous, the former by construction.  Now suppose ${\sigma}_a$ is open for all $\twokob$.  By Thom Transversality the subspace of $C^{\infty}(M(a),N(a))$ consisting of maps transverse to $S(a)$ is open and dense.  As ${\sigma}_a$ is open, the pre-image under ${\sigma}_a$ of this subspace must also be open and dense in $C^{\infty}_{\brac}(M,N)$, and so the intersection of these subspaces over all $\twokob$ is also open and dense in $C^{\infty}_{\brac}(M,N)$, but this space is precisely $T_SC^{\infty}_{\brac}(M,N)$.  Thus it remains to show
\begin{minilemma}
${\sigma}_a\textit{ is open for every }\twokob$\end{minilemma}

\begin{proof} Choose a Riemannian metric on $N$.  This metric induces metric on all of the function spaces in our discussion.   Let $f\;{\in}\;C^{\infty}_{\brac}(M,N)$.  It suffices to show for any ${\epsilon}>0$ there exists a ${\delta}>0$ such that for any $h\;{\in}\;C^{\infty}(M(a),N(a))$ with $d(h,{\sigma}_a(f))<{\delta}$ there exists $H\;{\in}\;C^{\infty}_{\brac}(M,N)$ satisfying ${\sigma}_a(H)=h$ and $d(H,f)<{\epsilon}$.  As $N$ is Riemannian and compact, there exists a ${\delta}_0>0$ for which balls of radius ${\delta}_0$ are geodesic.  Identify geodesic collar neighborhoods of $M(a)$ in $M(1)$ and $N(a)$ in $N(1)$ with $M(a){\times}[0,{\delta}_1)^k(1-a)$ with ${\delta}_1$ chosen small enough to ensure ${\delta}_1<\text{min}({\delta}_0,\frac{\epsilon}{3})$ and ($f(m,\vec{x}),f(m,\vec{0})$)$<\frac{\epsilon}{3}$ for all $\vec{x}\;{\in}\;[0,1)^k(1-a)$.   Now suppose $h\;{\in}\;C^{\infty}(M(a),N(a))$ with $d(h,{\sigma}_a(f))<{\delta}_1$.  Then define
$$H:M(1){\rightarrow}N(1)$$
$$m{\mapsto}f(m)\text{ if }m{\notin}M(a){\times}[0,{\delta}_1)^k(1-a)$$
$$m=(m',\vec{x}){\mapsto}r(|\vec{x}|)f(m)+(1-r(|\vec{x}|))(h(m'),\vec{x})\text{ if }m{\in}M(a){\times}[0,{\delta}_1)^k(1-a)$$
where $r:I{\rightarrow}I$ is a monotone smooth function satisfying $r([0,\frac{1}{3}])=0$ and $r([\frac{2}{3},1])=1$.
The definition makes sense since we've arranged for $d(f(m),(h(m'),\vec{x}))<{\delta}_0$.  By setting $\delta={\delta}_1$ we have found $\delta$ with the desired property.\end{proof}
This also concludes the proof of lemma \ref{translemma}.\end{proof}

\begin{append}\label{appendix3}$\textbf{ - Ehresman Fibration Theorem for } \brac \textbf{-manifolds}$\end{append}

\begin{lemma}Let $M$ be a neat $d$-dimensional $\brac$-submanifold of $\mathbb{R}^k_+{\times}\mathbb{R}^N$ for some $N,k{\geq}0$, let $X$ be a plain old closed l-dimensional manifold promoted to a $\brac$-space and suppose $f:M\;{\rightarrow}\;X$ is a smooth $\brac$-map such that each $f(a)$ is a proper submersion.  Then $f$ is a fiber bundle with a $\brac$-manifold as its fiber.\end{lemma}

\begin{proof}We convert this situation to one where the usual Ehresman
fibration theorem applies.  To that end, there is some $\epsilon$ for
which the collar of $M$ is perpendicular.  Pick a smooth $\textit{even}$ function ${\sigma}:\mathbb{R}\;{\rightarrow}\;\mathbb{R}_+$ which  restricted to $\mathbb{R}$ is a monotone increasing function such that $\sigma(x)$=0 if $x\;{\leq}\;\frac{\epsilon}{3}$ and ${\sigma}(x)=x$
if $x\;{\geq}\;\frac{{2}\epsilon}{3}$.  Now define
 
$${\Sigma}:={\sigma}^k{\times}id_{\mathbb{R}^N}:\mathbb{R}^k{\times}\mathbb{R}^{N}\;{\longrightarrow}\;\mathbb{R}^{k}_+\times\mathbb{R}^{N}$$

Now set $\overline{M}:={\Sigma}^{-1}(M)$.  Then $\overline{M}$ is a closed submanifold of $\mathbb{R}^{k+N}$ and $f{\circ}{\Sigma}:\overline{M}{\rightarrow}X$ is a proper submersion in the usual sense.  Now the usual Ehresman's Fibration Theorem applies, letting us conclude that $f\circ{\Sigma}$ is a fiber bundle, with fiber a closed $d$-$l$ dimensional manifold.  By construction we see that $f$ must also be a fiber bundle whose fibers are $d$-$l$ $\brac$-submanifolds of $\mathbb{R}^k_+{\times}\mathbb{R}^N$.  Note that the fibers of $f\circ{\Sigma}_{|M}$ are neat $\brac$-submanifolds, but we are not claiming the fibers of $f$ are neat however.\end{proof}

\begin{append}\label{appendix4}$\textbf{ - Phillips extended to k-spaces}$\end{append}

Let $M$ and $N$ be $\brac$-manifolds.

\begin{defn}Let $\text{Sur}_{\brac}(TM,TN)$ denote the space of $\brac$-bundle maps 
$TM{\rightarrow}TN$ such that each $TM(a)\;{\rightarrow}\;TN(a)$ is a fiberwise surjection.\end{defn}

\begin{defn}Let $\text{Sub}_{\brac}(M,N)$ be the space of $f{\in}C^{\infty}_{\brac}(M,N)$ such that $Df{\in}\text{Sur}(M,N)$.\end{defn}

For us, an open manifold is a manifold with no closed components.
When $k=0$, Philips's Submersion Theorem states that the map 
$${\text{Sub}_{\brac}(M,N)}\;{\rightarrow}\;\text{Sur}_{\brac}(TM,TN)$$
$$f{\mapsto}Df$$ is a weak equivalence.\\

We shall inductively extend this result to $\brac$-manifolds.  The inductive
hypothesis that seems to work naturally involves a broader class of $\brac$-spaces than $\brac$-manifolds.

\begin{defn} A $\brac$-space $M$ type 17 if:\\
i) $M(1)$ is a manifold.\\
ii) $M(a)$ is a submanifold of $M(b)$ and $M(a<b)$ is a closed cofibration for all $\twokmor$.\end{defn}

Every $\brac$-manifold is a type 17 $\brac$-space.  One important feature type 17 $\brac$-spaces possess that $\brac$-manifolds do not is that they are closed under the $\overline{{{\partial}_i}}$ operation.  We will use this fact.\\

\begin{prop} Let $N$ be a $\brac$-manifold, and let $M$ be a type 17 $\brac$-space such that $M(a)$ is an open manifold for each $\twokob$.  Then
$$\text{Sub}(M,N)\;{\longrightarrow}\;\text{Sur}(TM,TN)$$
$$f{\mapsto}Df$$
is a weak equivalence.\end{prop}

\begin{proof}
Induction on k.  The base case $k=0$ is dealt with by Philips's Submersion Theorem $\cite{Philips}$.  Now assume the proposition true for $k-1$.  Let us temporarily introduce
   $$X(a,b):=C^{\infty}(M(a),N(b))\text{ if }a<b\text{ and}$$
   $$X(a,a):=\text{Sub}(M(a),N(a)).$$ 
Similarly, let
    $$Y(a,b):=\text{Maps}(TM(a),TN(b))\text{ if }a<b\text{ and}$$
   $$Y(a,a):=\text{Sur}(TM(a),TN(a)).$$
Now proceeding with the proof,
$$\text{Sub}_{\brac}(M,N)=\text{lim}_{(a,b){\in}P^k}X(a,b)$$
now we expand the limit out
$$=\text{lim}(\text{lim}_{(a,b){\in}P^{k-1}}{\partial}_kX(a,b){\rightarrow}\text{lim}_{(a,b){\in}P^{k-1}}X(a,b+e_k){\leftarrow}\text{lim}_{(a,b){\in}P^{k-1}}\overline{{\partial}_k}X(a,b))$$
The map on the right is induced by restriction to $M(1$-$e_k)$ and is 
a fibration so we may replace the limit by a homtopy limit without affecting
homotopy type.
$${\simeq}\text{holim}(\text{lim}_{(a,b){\in}P^{k-1}}{\partial}_kX(a,b){\rightarrow}\text{lim}_{(a,b){\in}P^{k-1}}X(a,b+e_k){\leftarrow}\text{lim}_{(a,b){\in}P^{k-1}}\overline{{\partial}_k}X(a,b))$$
Now by induction on k,
$${\simeq}\text{holim}(\text{lim}_{(a,b){\in}P^{k-1}}{\partial}_kY(a,b){\rightarrow}\text{lim}_{(a,b){\in}P^{k-1}}X(a,b+e_k){\rightarrow}\text{lim}_{(a,b){\in}P^{k-1}}\overline{{\partial}_k}Y(a,b))$$
The middle term may also be dealt with.  Because the spaces in the middle limit are entirely of the form $C^{\infty}(M(a),N(b))$ and the map
$$C^{\infty}(M(a),N(b)){\rightarrow}\text{Maps}(TM(a),TN(b))$$ is a weak equivalence, by an inductive argument similar to the one we are in the process of giving here, we may conclude that
$$\text{lim}_{(a,b){\in}P^{k-1}}X(a,b+e_k){\simeq}\text{lim}_{(a,b){\in}P^{k-1}}Y(a,b+e_k)$$
and hence
$${\simeq}\text{holim}(\text{lim}_{(a,b){\in}P^{k-1}}{\partial}_kY(a,b){\rightarrow}\text{lim}_{(a,b){\in}P^{k-1}}Y(a,b+e_k){\leftarrow}\text{lim}_{(a,b){\in}P^{k-1}}\overline{{\partial}_k}Y(a,b))$$
$${\simeq}\text{lim}(\text{lim}_{(a,b){\in}P^{k-1}}{\partial}_kY(a,b){\rightarrow}\text{lim}_{(a,b){\in}P^{k-1}}Y(a,b+e_k){\leftarrow}\text{lim}_{(a,b){\in}P^{k-1}}\overline{{\partial}_k}Y(a,b))$$
$$=\text{Sur}_{\brac}(TM,TN).$$
\end{proof}

\begin{append}\label{appendix5}$\textbf{ - }\langle\textbf{k}\rangle\textbf{-vector bundles}$\end{append}

In GMTW they deal with the following situation:  Let $M$ be a manifold, and $U,V$ vector bundles over $M$.  Let $\Gamma(U,V)$ denote the space of vector bundle isomorphisms from $U$ to $V$.  Stabilizing the vector bundles yields a map $\sigma:\Gamma(U,V){\rightarrow}\Gamma(U{\oplus}\epsilon,V{\oplus}\epsilon)$.  $\sigma$ is $\text{dim}U-\text{dim}M-1$ connected.\\

To make this precise: for $\twokmor$ let $\Gamma_{a}(b)$ denote the space of bundle isomorphisms $M(a<b)^*U(b){\rightarrow}M(a<b)^*V(b)$ over $M(a)$.  Clearly there are restriction maps $Res_{a<b}:\Gamma_{b}(b){\rightarrow}\Gamma_{a}(b)$ and there are inclusion maps $In_{a<b}:\Gamma_{a}(a){\rightarrow}\Gamma_{a}(b)$ coming from the chosen splitting.\\

Let $P_n=\underline{2}^k{\times}\underline{2}^{k^{\text{op}}}$, and for $(a,c)$ and $(b,d)$ with $a<b$ and $c>d$ we define $\Gamma^{\brac}((a,c)<(b,d))=Res_{d<c}{\circ}In_{a<b}$ and then $\Gamma^{\brac}(U,V):=\text{lim}_{P_n}\Gamma_a(c)$.  Let $d(a)=\text{dim}U(a)-\text{dim}M(a)$, and set $d:=\text{min}_{a{\in}\underline{2}^k}d(a)$.\\

\begin{lemma}
Let $M$ be a $\brac$-CW space such that $\text{dim}M(a)<\text{dim}M(b)$ for all $\twokmor$ (certainly all $\brac$-manifolds satisfy this criterion).  Let $U$,$V$,$\epsilon$ be geometric $\brac$-vector bundles over $M$, such that $\epsilon(a)$ is trivial for all $\twokob$.  The stabilization map
$$\sigma:\Gamma^{\brac}(U,V)\rightarrow\Gamma^{\brac}(U\oplus{\epsilon},V\oplus{\epsilon})$$
is $d-1$ connected.\end{lemma}

\begin{proof}  First we need a mini-lemma.
\begin{minilemma} Let $\mathcal{D}$ be the category $(A{\rightarrow}B{\leftarrow}{C})$, $F$, and $G$ functors from $\mathcal{D}$ to spaces.  Suppose $t$ is a natural transformation such that $t(A)$, $t(B)$, and $t(C)$ are i,k, and j connected respectively with i$<$k.  Then the induced map $t':\text{holim}F\rightarrow\text{holim}G$ is min$\{i,j\}$ connected.\end{minilemma}

\begin{proof}
Repeated use of the five-lemma.\end{proof}

The proof is by induction on $\brac$.  The base case $k=0$ boils down to the original result from $\cite{GMTW}$.
Now assume the result is true for all $\langle{k-1}\rangle$ manifolds.  By subdividing the limit,
$$\Gamma^{\brac}(U,V)$$
$$=\text{lim}(\text{lim}_{P_{k-1}}\Gamma_{a+e_{k}}(c+e_{k})\rightarrow\text{lim}_{P_{k-1}}\Gamma_{a}(c+e_{k})\leftarrow{\text{lim}}_{P_{k-1}}\Gamma_{a}(c))$$
The first arrow is a fibration so we may replace the outer $\text{lim}$ with $\text{holim}$, or in this simple case, the homotopy pullback, and preserve the homotopy type.
$$=\text{holim}(\text{lim}_{P_{k-1}}\Gamma_{a+e_{k}}(c+e_{k})\rightarrow\text{lim}_{P_{k-1}}\Gamma_{a}(c+e_{k})\leftarrow{\text{lim}}_{P_{k-1}}\Gamma_{a}(c))$$
By induction, the stabilization map on $\text{lim}_{P_{k-1}}\Gamma_{a+e_{k}}(c+e_{k})$ and $\text{lim}_{P_{k-1}}\Gamma_{a}(c))$ is $d_1$ and $d_2$ connected respectively, but
$\text{lim}_{P_{k-1}}\Gamma_{a}(c+e_{k})$ is $d_3>d_1$ connected since for every $a{\in}\underline{2}^{k-1}$ we have
$$\text{dim}Res_{a}U(a+e_k)-M(a)>\text{dim}U(a+e_k)-M(a+e_k)$$
Now we may use the mini-Lemma to conclude that the stabilization map on $\Gamma^{\brac}(U,V)$ is min$\{d_1,d_2\}$ connected, but min$\{d_1,d_2\}=\text{min}_{a{\in}\underline{2}^k}d(a)$.\end{proof}

\bibliographystyle{alpha}
\bibliography{k-cobordism}

\end{document}